\documentclass[11pt,reqno]{amsart}
\usepackage[utf8]{inputenc}
\usepackage[T1]{fontenc}
\usepackage{lmodern}
\usepackage{amsfonts,amsthm,amsmath,amssymb}
\usepackage{graphicx}
\usepackage[dvipsnames,table]{xcolor}   % »table« optons loads »colortbl« package
\usepackage{enumerate}
\usepackage{hyperref}
\usepackage{color}
\usepackage{tikz-cd}
\usepackage{subcaption}
\usepackage[margin=1in]{geometry}
\usepackage[
    maxbibnames=99,
    backend=biber,
    style=alphabetic,
    sorting=nyt,
    giveninits=true
]{biblatex}
\DeclareFieldFormat{pages}{#1}
\renewbibmacro{in:}{%
  \ifentrytype{article}
    {}
    {\bibstring{in}%
     \printunit{\intitlepunct}}}
\DeclareFieldFormat
[article,inbook,incollection,inproceedings,patent,thesis,unpublished]
  {title}{\mkbibemph{#1}}
\DeclareFieldFormat{journaltitle}{#1\isdot}
\DeclareFieldFormat[article]{volume}{\mkbibbold{#1}}
\DeclareFieldFormat[article]{number}{\bibstring{number}\addnbspace #1}

\renewbibmacro*{journal+issuetitle}{%
  \usebibmacro{journal}%
  \setunit*{\addspace}%
  \iffieldundef{series}
    {}
    {\newunit
     \printfield{series}%
     \setunit{\addspace}}%
  \printfield{volume}%
  \setunit{\addspace}%
  \usebibmacro{issue+date}%
  \setunit{\addcomma\space}%
  \printfield{number}%
  \setunit{\addcolon\space}%
  \usebibmacro{issue}%
  \setunit{\addcomma\space}%
  \printfield{eid}
  \newunit}

\newtheoremstyle{mystyle}%                % Name
  {}%                                     % Space above
  {}%                                     % Space below
  {\itshape}%                                     % Body font
  {}%                                     % Indent amount
  {\bfseries}%                            % Theorem head font
  {.}%                                    % Punctuation after theorem head
  { }%                                    % Space after theorem head, ' ', or \newline
  {\thmname{#1}\thmnumber{ #2}\thmnote{ (#3)}}%                                     % Theorem head spec (can be left empty, meaning `normal')

\theoremstyle{mystyle}
\newtheorem{Thm}{Theorem}[section]
\newtheorem{Lem}[Thm]{Lemma}
\newtheorem{Cor}[Thm]{Corollary}
\newtheorem{Prop}[Thm]{Proposition}

\newtheorem{Que}[Thm]{Question}

\theoremstyle{definition}
\newtheorem{Def}[Thm]{Definition}
\newtheorem{Ex}[Thm]{Example}

\newtheorem{Obs}[Thm]{Observation}

\theoremstyle{remark}
\newtheorem{Rmk}[Thm]{Remark}

\newcommand{\R}{\mathbb{R}}
\newcommand{\Z}{\mathbb{Z}}

\DeclareMathOperator{\Int}{int}
\newcommand{\HH}{\mathcal H}

\newcommand{\TT}{\mathcal T}
\newcommand{\del}{\partial}

\author{Sylvain Courte}
\address{Institut Fourier, Universit\'e Grenoble Alpes,
100 rue des maths, 38610 Gi\`eres, France}
\email{\href{mailto:sylvain.courte@univ-grenoble-alpes.fr}{sylvain.courte@univ-grenoble-alpes.fr}}

\author{Delphine Moussard}
\address{Institut Fourier, Universit\'e Grenoble Alpes,
100 rue des maths, 38610 Gi\`eres, France}
\email{\href{mailto:delphine.moussard@univ-grenoble-alpes.fr}{delphine.moussard@univ-grenoble-alpes.fr}}

\author{Qiuyu Ren}
\address{Department of Mathematics, University of California, Berkeley, Berkeley, CA 94720, USA}
\email{\href{mailto:qiuyu_ren@berkeley.edu}{qiuyu\_ren@berkeley.edu}}

\author{Xiaozhou Zhou}
\address{Department of Mathematical Sciences, Ritsumeikan University, 1-1-1 Nojihigashi, Kusatsu, Shiga 525-
8577, Japan}
\email{\href{ra0108ik@ed.ritsumei.ac.jp}{ra0108ik@ed.ritsumei.ac.jp}}

\title{Multisections and bridge positions in arbitrary dimensions}

\begin{document}

\begin{abstract}
Multisections were defined by Ben Aribi--Courte--Golla--Moussard as a way to decompose closed manifolds into $1$-handlebodies, which is a generalization of Heegaard splittings and trisections. Previously, their existence was only known in dimensions up to $5$. We show that multisections exist for closed manifolds in arbitrary dimensions. We also show that submanifolds of codimension at least $2$ in multisected manifolds can always be put into an appropriate bridge position, generalizing the existence result on bridge trisections in dimension $4$.
\end{abstract}

\maketitle

\section{Introduction}\label{sec:intro}
The notion of multisections of smooth manifolds was introduced in \cite{aribi2023multisections} as a generalization of Heegaard splittings \cite{heegaard1898forstudier} in dimension $3$ and of trisections in dimension $4$ \cite{gay2016trisecting}. Roughly speaking, an \textit{$(n-1)$-section} (or \textit{multisection} for short) of a closed smooth $n$-manifold $X$ is a decomposition $X=X_1\cup\cdots\cup X_{n-1}$ such that for every index set $\emptyset\ne I\subsetneq\{1,\cdots,n-1\}$, the intersection $X_I=\cap_{i\in I}X_i$ is a smooth manifold (with corners) built only out of $0,1$-handles of the appropriate dimension, and $\cap_{i=1}^{n-1}X_i$ is a closed surface. A $2$-section of a (closed, smooth) $3$-manifold is called a \textit{Heegaard splitting}, a $3$-section of a $4$-manifold is called a \textit{trisection}, and a $4$-section of a $5$-manifold is called a \textit{quadrisection}. The existence of multisections for arbitrary closed oriented smooth manifolds of dimensions $3,4,5$ is proved in \cite{heegaard1898forstudier}, \cite[Theorem~4]{gay2016trisecting}, \cite[Theorem~7.3]{aribi2023multisections}, respectively, but higher dimensional cases were previously unknown. The first goal of our paper is to prove the existence of multisections in arbitrary dimensions.

\begin{Thm}\label{thm:multisection}
Every closed smooth $n$-manifold admits an $(n-1)$-section.
\end{Thm}

A key step in the proof of Theorem~\ref{thm:multisection} is related to the notion of \textit{bridge position}. Roughly speaking, we say that a closed submanifold $S$ in a multisection $X=X_1\cup\cdots\cup X_{n-1}$ is in \textit{bridge position} if it is transverse to every sector (and subsector), and, for every index set $\emptyset\ne I\subsetneq\{1,\cdots,n-1\}$, the submanifold $S_I:=S\cap X_I\subset X_I$ is a collection of $\partial$-parallel disks. When $X=S^3$ with the genus-$0$ Heegaard splitting and $\dim S=1$, this recovers the classical notion of bridge position for links. When $\dim X=4$, $\dim S=2$, this recovers the notion of bridge trisections for surface knots in $4$-manifolds as introduced by Meier--Zupan \cite{meier2017bridge,meier2018bridge}, whose existence was carefully established there. Existence was also proved for $X=S^5$ with the genus-$0$ quadrisection and $\dim S=3$ by Aranda--Blackwell--Kim--Naylor--Pongtanapaisan \cite{aranda2026bridge}. The second goal of our paper is to prove the existence of bridge position in arbitrary dimensions.

\begin{Thm}\label{thm:bridge}
If $X=X_1\cup\cdots\cup X_{n-1}$ is a multisection and $S\subset X$ is a closed submanifold of codimension at least $2$, then $S$ can be isotoped to be in bridge position.
\end{Thm}

In fact, we will also prove the existence of bridge positions in a relative setup, where $X$ may have boundary and is multisected in an appropriate sense, and where $S$ is a properly embedded submanifold with codimension at least $2$. The condition on codimension is necessary, as can already be seen in dimension $3$.

In particular, in dimensions $4$ and $5$, Theorem~\ref{thm:multisection} and Theorem~\ref{thm:bridge} yield alternative proofs of the existence of trisections \cite{gay2016trisecting}, quadrisections \cite{aribi2023multisections}, bridge trisections \cite{meier2017bridge,meier2018bridge}, and bridge quadrisections in $S^5$ \cite{aranda2026bridge}. In contrast to the proofs in the cited works, our argument is symmetric with respect to the sectors, lending it a certain aesthetic appeal.
Moreover, our proof is very general: it works in every dimension, for smooth or PL manifolds, orientable or not.

\medskip
The paper is organized as follows. In the preliminary section~\ref{sec:def_closed}, we give the definitions of multisections and bridge positions for closed manifolds and submanifolds. In Section~\ref{sec:multisection}, we prove Theorem~\ref{thm:multisection} using an inductive argument assuming Theorem~\ref{thm:bridge}. The proof of Theorem~\ref{thm:bridge} is an analogous inductive argument, presented in Section~\ref{sec:bridge}, but which requires relaxing our setup to the case of manifolds with boundary and properly embedded submanifolds, which we introduce in Section~\ref{sec:def_general}. Section~\ref{sec:bridge} ends with a proof that the handle decompositions we use also exist for PL manifolds. In Section~\ref{sec:trisection_5}, partly as an illustration of our proof method, we reprove a result by Lambert-Cole--Miller which asserts that $5$-manifolds admit trisections that agree with a given trisection on the boundary. The reader may read Section~\ref{sec:trisection_5} before Section~\ref{sec:multisection} for the flavor of the proof of Theorem~\ref{thm:multisection}. Finally, Section~\ref{sec:future} proposes some perspectives about uniqueness of multisections and bridge positions, and generalizations of multisections.

\section*{Acknowledgments}
This project was initiated during the \textit{Multisectors} workshop at CIRM in Luminy, France, July 2025. The authors thank CIRM for its hospitality and for making this collaboration possible. We are also grateful to David Gay for organizing the workshop and for many interesting discussions. This research was partially conducted during the period when QR served as a Clay Research Fellow.

\section{Definitions: Closed case}\label{sec:def_closed}
In this paper, unless otherwise stated, every manifold is smooth, possibly with boundary and corners. We refer the reader to \cite[Chapter~4 and \S~6.2]{gompf20234} for background on handle decompositions of manifolds and ambient handle decompositions of submanifolds.

\subsection{Multisection}\label{sbsec:def_multisection_closed}
A \textit{$1$-handlebody} of dimension $n$ is a manifold built out of $n$-dimensional $0$- and $1$-handles. Equivalently, it is a manifold of the form $$H=\sqcup_{i=1}^k(\natural^{\ell_i}(S^1(\tilde\times)D^{n-1})),\ k\ge0,\ \ell_i\ge0,$$ where $(\tilde\times)$ denotes either $\times$ or $\tilde\times$, affecting the orientability of the connected components of $H$. The $1$-handlebody is orientable if and only if no $S^1\tilde\times D^{n-1}$ summand is present.

\begin{Def}[Multisection, {\cite{aribi2023multisections}}]\label{def:multisection}
Let $n\ge2$ be an integer. An \textit{$(n-1)$-section} (or \textit{multisection}) of a closed $n$-manifold $X$ is a decomposition $X=X_1\cup\cdots\cup X_{n-1}$, so that if we set $X_I:=\cap_{i\in I}X_i$ for $\emptyset\ne I\subset\{1,\cdots,n-1\}$ (these are called \textit{sectors} of the multisection), the following conditions hold:
\begin{enumerate}[(1)]
\item Each $X_I\subset X$ is a submanifold with corners, whose codimension-$k$ stratum is $$\bigcup_{J\supset I,|J|-|I|=k}\Int(X_J).$$
\item For each $I\ne\{1,\cdots,n-1\}$, after smoothing the corners, $X_I$ is diffeomorphic to a $1$-handlebody of dimension $n+1-|I|$.
\end{enumerate}
\end{Def}

For later purposes, it is convenient to turn the handlebody upside down and obtain the following alternative description of the second item in Definition~\ref{def:multisection}.
\begin{enumerate}[(1')]
\setcounter{enumi}{1}
\item For each $I\ne\{1,\cdots,n-1\}$, after smoothing the corners, $X_I$ admits a handle decomposition rel $\partial X_I$ with handles of coindices $0,1$ (i.e. indices $\dim X_I$, $\dim X_I-1$).
\end{enumerate}

If $X=X_1\cup\cdots\cup X_{n-1}$ is a multisection, $\Sigma:=X_{\{1,\cdots,n-1\}}\subset X$ is a smooth closed surface, sometimes called the \textit{multisection surface} of the multisection. One can easily show that the inclusions induce bijections  $\pi_0(\Sigma)\xrightarrow{\cong}\pi_0(X)$ and $\pi_0(\Sigma)\xrightarrow{\cong}\pi_0(X_I)$ for all $I$; thus, one may restrict attention, without loss of generality, to connected manifolds and connected $1$-handlebodies, although we will not be imposing this assumption. Analogously, one can show that $X$ and all of the sectors $X_I$ have the same orientability.

For a nonempty finite set $I$, let $\Delta^I$ denote the standard $(|I|-1)$-simplex with vertices given by elements of $I$. Its faces are $\Delta^J$ for various $J\subset I$. Let $0$ denote the center of $\Delta^I$, and $\Delta^I_J\subset\Delta^I$ be the cone on $\Delta^{I\backslash J}$ with cone point $0$, for $J\subsetneq I$. Thus, $\Delta^I_\emptyset=\Delta^I$. The notation $\Delta^I_I$ corresponds to a degenerate case and will not be used.

If $X=X_1\cup\cdots\cup X_{n-1}$ is a multisection, $\emptyset\ne I\subset\{1,\cdots,n-1\}$, then $\Delta^I$ is a local transverse model for $X_I$ \cite[Lemma~2.4]{aribi2023multisections}. More precisely, a closed tubular neighborhood of $X_I$ in the codimension-$0$ submanifold $\cup_{i\in I}X_i$ is diffeomorphic to $\Delta^I\times X_I$ via a diffeomorphism that respects the decomposition into sectors: its intersection with $X_J$ is $\Delta^I_J\times X_I$, for all $J\subset I$.

\subsection{Bridge position}
Suppose $M$ is a manifold with boundary and $N\subset M$ is a properly embedded submanifold. Let $P=[0,1]\times N$ with the corner along $\{1\}\times\partial N$ smoothened, $\partial_-P=\{0\}\times N$, and $\partial_+P=\overline{\partial P\backslash\partial_-P}$. We say $N$ is \textit{$\partial$-parallel} in $M$ if there is an embedding $P\subset M$ as a submanifold such that $P\cap\partial M=\partial_-P$ and $\partial_+P=N$.

\begin{Def}[Bridge position]\label{def:bridge}
Let $X=X_1\cup\cdots\cup X_{n-1}$ be a multisection of a closed $n$-manifold. A closed submanifold $S\subset X$ is said to be in \textit{bridge position} if the following conditions hold:
\begin{enumerate}[(1)]
\item $S$ is transverse to each sector $X_I$; in particular, $S_I:=S\cap X_I$ is a properly embedded submanifold of $X_I$ with corners.
\item For each $I\ne\{1,\cdots,n-1\}$, after smoothing the corners, $S_I\subset X_I$  is $\partial$-parallel and is a disjoint union of disks of codimension $\mathrm{codim}_XS$.
\end{enumerate}
\end{Def}
It is convenient to note the following alternative description of the second condition in Definition~\ref{def:bridge} in terms of ambient handle decompositions rel boundary (see \cite[\S6.2]{gompf20234} for terminology).
\begin{enumerate}[(1')]
\setcounter{enumi}{1}
\item For each $I\ne\{1,\cdots,n-1\}$, after smoothing the corners, $S_I\subset X_I$ is contained in a collar neighborhood of $\partial X_I$, in which it admits an ambient handle decomposition rel $\partial S_I\subset\partial X_I$ with handles of coindex $0$ (i.e. index $\dim S_I$).
\end{enumerate}

\section{Multisections}\label{sec:multisection}
In this section, assuming Theorem~\ref{thm:bridge}, we prove Theorem~\ref{thm:multisection}, the existence of multisections for closed manifolds in arbitrary dimensions.

The strategy is as follows. We start with any decomposition of an $n$-manifold $X$ into $n-1$ pieces, $X=X_1\cup\cdots\cup X_{n-1}$ (e.g. $X_1=X$, $X_2=\cdots=X_{n-1}=\emptyset$), such that $X_I=\cap_{i\in I}X_i$ is a submanifold with corners of the appropriate dimension for every index set $I$; for the purpose of this section only, we call such a decomposition $X=X_1\cup\cdots\cup X_{n-1}$ a \textit{pre-multisection} of $X$. In Section~\ref{sbsec:complexity}, we define the \textit{complexity} of a pre-multisection, valued in a well-ordered set, so that the pre-multisection has complexity $0$ (the unique minimum) if and only if it is a multisection of $X$. Then, in Sections~\ref{sbsec:modification} and~\ref{sbsec:proof_decrease}, we show that it is always possible to modify a pre-multisection to decrease its complexity if the complexity is nonzero. This implies that any pre-multisection of $X$ can be modified to a multisection of $X$ in finitely many steps, proving Theorem~\ref{thm:multisection}.

\subsection{Complexity of a pre-multisection}\label{sbsec:complexity}
We first define the complexity of compact manifolds with (possibly empty) boundary. For later convenience, instead of considering absolute handle decompositions of manifolds, we will turn them upside down and consider handle decompositions that build a manifold from its boundary. Thus, a $1$-handlebody of dimension $m$ is a manifold with a rel boundary handle decomposition that has only $(m-1)$- and $m$-handles.

\begin{Def}
The \textit{complexity} of a compact $m$-manifold $M$ with (possibly empty) boundary is $$c(M)=\min\{(\#(m-2)\text{-handles},\#(m-3)\text{-handles},\cdots,\#0\text{-handles})\}\in\Z_{\ge0}^{m-1},$$ where $\Z_{\ge0}^{m-1}$ is equipped with the lexicographical order, and the minimum is taken over all handle decompositions of $M$ rel $\partial M$. The \textit{complexity} of a compact manifold with corners is the complexity of its (corner-)smoothing.
\end{Def}

\begin{Def}
The \textit{complexity} of a pre-multisection $X=X_1\cup\cdots\cup X_{n-1}$ of a closed $n$-manifold $X$ is $$c(X_1,\cdots,X_{n-1})=\left(\sum_{|I|=1}c(X_I),\sum_{|I|=2}c(X_I),\cdots,\sum_{|I|=n-2}c(X_I)\right)\in\Z_{\ge0}^{n(n-1)/2-1}.$$
\end{Def}

Theorem~\ref{thm:multisection} is implied by the following proposition.

\begin{Prop}\label{prop:modification}
Let $c_r(\cdot)$ denote the $r$-th coordinate of $c(\cdot)$. If $X=X_1\cup\cdots\cup X_{n-1}$ is a pre-multisection of $X$, and $r$ is the maximal index with $c_r(X_1,\cdots,X_{n-1})>0$, then there exists a pre-multisection $X=X_1'\cup\cdots\cup X_{n-1}'$ such that $c_\ell(X_1',\cdots,X_{n-1}')\le c_\ell(X_1,\cdots,X_{n-1})$ for $\ell\le r$, with strict inequality when $\ell=r$.
\end{Prop}

\subsection{The complexity-decreasing modification}\label{sbsec:modification}
In this section, we propose the appropriate modification from $(X_1,\cdots,X_{n-1})$ to $(X_1',\cdots,X_{n-1}')$ for proving Proposition~\ref{prop:modification}. Fix (rel boundary) handle decompositions on the sectors realizing the complexity $c(X_1,\cdots,X_{n-1})$. By assumption, for some $3\le m\le n$, $0\le k\le m-2$, some $m$-dimensional sector $X_I$ has a $k$-handle, and no $m'$-dimensional sector has a $k'$-handle if $(m',k')<(m,k)$ in lexicographical order. In particular, the $m'$-dimensional sectors with $3\leq m'<m$ are all $1$-handlebodies. We modify $X=X_1\cup\cdots\cup X_{n-1}$ to remove one $k$-handle in $X_I$, so that no sector $X_J$ with $|J|\le|I|$ gets more handles with coindices greater than $1$.

To start, we fix some notation. Without loss of generality, say $I=\{m-1,m,\cdots,n-1\}$. Let $h$ be a $k$-handle in $X_I$, $B=B(h)$ be its core, $S=S(h)=\partial B(h)$ be its attaching sphere. Thus, $B$ is a $k$-disk, $S$ is a $(k-1)$-sphere, and $h\cong_\alpha D^{m-k}\times B$. Let $\nu=\nu(h)\cong_\beta D^{n-m}\times h$ be a closed tubular neighborhood of $h$ in the codimension-$0$ submanifold $\cup_{i\in I}X_i$. Here, the diffeomorphisms $\alpha,\beta$ are fixed once and for all.

By assumption, $\partial X_I=\cup_{i=1}^{m-2}X_{\{i\}\cup I}$ is a multisection of $\partial X_I$ (after corner-smoothing). Since $S\subset\partial X_I$ has codimension $(m-1)-(k-1)=m-k\ge2$, Theorem~\ref{thm:bridge} allows us to isotope $S$ into bridge position; one can drag the handle $h$ along this isotopy without changing $X_I$. Henceforth, $S$ is in bridge position in $\partial X_I$, and we write $S_J=S\cap X_{J\cup I}$ for $J\subset I^c=\{1,\cdots,m-2\}$ (when $J=\{i\}$ is a singleton, we simply write $S_i$ instead of $S_{\{i\}}$).

We extend the decomposition $S=S_1\cup\cdots\cup S_{m-2}$ to a decomposition of $B$ (the \textit{$1$-cone-off decomposition}) as follows. Write $B=([1/2,1]\times S)\cup\tfrac12B$, where $\tfrac12B$ is a ball of half the size of $B$. For $i>1$, let $B_i=[1/2,1]\times S_i$. Let $B_1=([1/2,1]\times S_1)\cup\tfrac12B$. We extend the decomposition $B=B_1\cup\cdots\cup B_{m-2}$ to $h=h_1\cup\cdots\cup h_{m-2}$, and then to $\nu=\nu_1\cup\cdots\cup\nu_{m-2}$, in the natural way using the product structures. One should think of $\nu$ as an $n$-dimensional $k$-handle attached to $\cup_{i=1}^{m-2}X_i$, with attaching region denoted by $\partial_{bot}\nu$. By shrinking $h$ and $\nu$ if necessary, the product structures $\alpha,\beta$ are assumed to be chosen so that the restriction of $\nu=\nu_1\cup\cdots\cup\nu_{m-2}$ to $\partial_{bot}\nu$ is given by $(\partial_{bot}\nu)_i=\partial_{bot}\nu\cap X_i$, $i=1,\cdots,m-2$.

\textbf{The $(n,m,k)$-modification} ($3\le m\le n$, $0\le k\le m-2$):
\begin{align*}
X_i'&:=X_i\cup\nu_i,\quad 1\le i\le m-2,\\
X_i'&:=\overline{X_i\backslash\nu},\quad m-1\le i\le n-1.
\end{align*}

\subsection{Proof of the complexity-decreasing property}\label{sbsec:proof_decrease}
We prove that the proposed $(n,m,k)$-modification decreases the complexity of the pre-multisection. To do this, we analyze the change in the handle decomposition in each sector $X_J$. It suffices to show that for every $|J|\le|I|$, $X_J'$ (admits a handle decomposition that) has no more $i$-handles than $X_J$ for $i\le\dim(X_J)-2$, and that $X_I'$ has fewer $k$-handles than $X_I$.

\textbf{Case 0}: $J=I$.

By construction, $X_I'=\overline{X_I\backslash h}$, effectively removing the $k$-handle $h$ from $X_I$. Thus, $X_I'$ has one fewer $k$-handle than $X_I$ (and the same number of handles of other indices).

\textbf{Case 1}: $J\subsetneq I$.

We have $X_J'=\overline{X_J\backslash(\nu\cap X_J)}$. Recall that a closed neighborhood of $X_I$ in $\cup_{i\in I}X_i$ is identified with $\Delta^I\times X_I$ together with the decomposition into sectors. Under this identification, near $X_I$, $X_J$ is locally $\Delta^I_J\times X_I$, and by taking $\nu$ to be thin in the transverse direction we may identify $\nu\cap X_J$ with $\tfrac12\Delta^I_J\times h$. We observe that $X_J'$ is a deformation retract of~$X_J$, obtained by ``pressing'' $0\times h$ into $\Delta^I_J\times X_I$ along the radial directions. In fact, this retract witnesses a diffeomorphism $X_J'\cong X_J$ once the corners are smoothened; hence, $X_J'$ has the same handle decomposition as $X_J$.

\textbf{Case 2}: $J\subset I^c$.

We have $X_J'=X_J\cup\nu_J$. The attachment of $\nu_J$ is a thickening of the attachment of $B_J$ along $S_J$. The following lemma, whose proof crucially depends on the fact that $S\subset\partial X_I$ is in bridge position, implies that $X_J'$ admits a (rel boundary) handle decomposition with the same number of handles of every index, except possibly for some additional handles with coindices $0,1$.

\begin{Lem}\label{lem:handle_color}
$B_J$ is obtained from (a thickening of) $S_J$ by attaching $0,1$-handles.
\end{Lem}
\begin{proof}
Since $S\subset\partial X_I$ is in bridge position, for every $\emptyset\ne J\subset I^c$, $S_J$ is a disjoint union of disks of dimension $k-|J|$.

In particular, since $S_1$ is a disjoint union of some $d$ disks, $B_1=([1/2,1]\times S_1)\cup\tfrac12B$ is obtained from $S_1$ by attaching one $0$-handle and $d$ $1$-handles (one can cancel a pair of $0,1$-handles if $d>0$). Similarly, if $\{1\}\subsetneq J$, then $B_J=([1/2,1]\times S_J)\cup(\{1/2\}\times S_{J\backslash\{1\}})$ is obtained from $S_J$ by attaching some $0,1$-handles. Finally, if $1\notin J$, then $B_J=[1/2,1]\times S_J$ is a thickening of $S_J$.
\end{proof}

\textbf{Case 3}: $J=J_1\cup J_2$, where $\emptyset\ne J_1\subsetneq I$, $\emptyset\ne J_2\subset I^c$. See Figure~\ref{fig:k-hd}.

\begin{figure}[ht]
    \centering
    \begin{subfigure}[t]{0.44\linewidth}
        \centering
        \includegraphics[width=\linewidth]{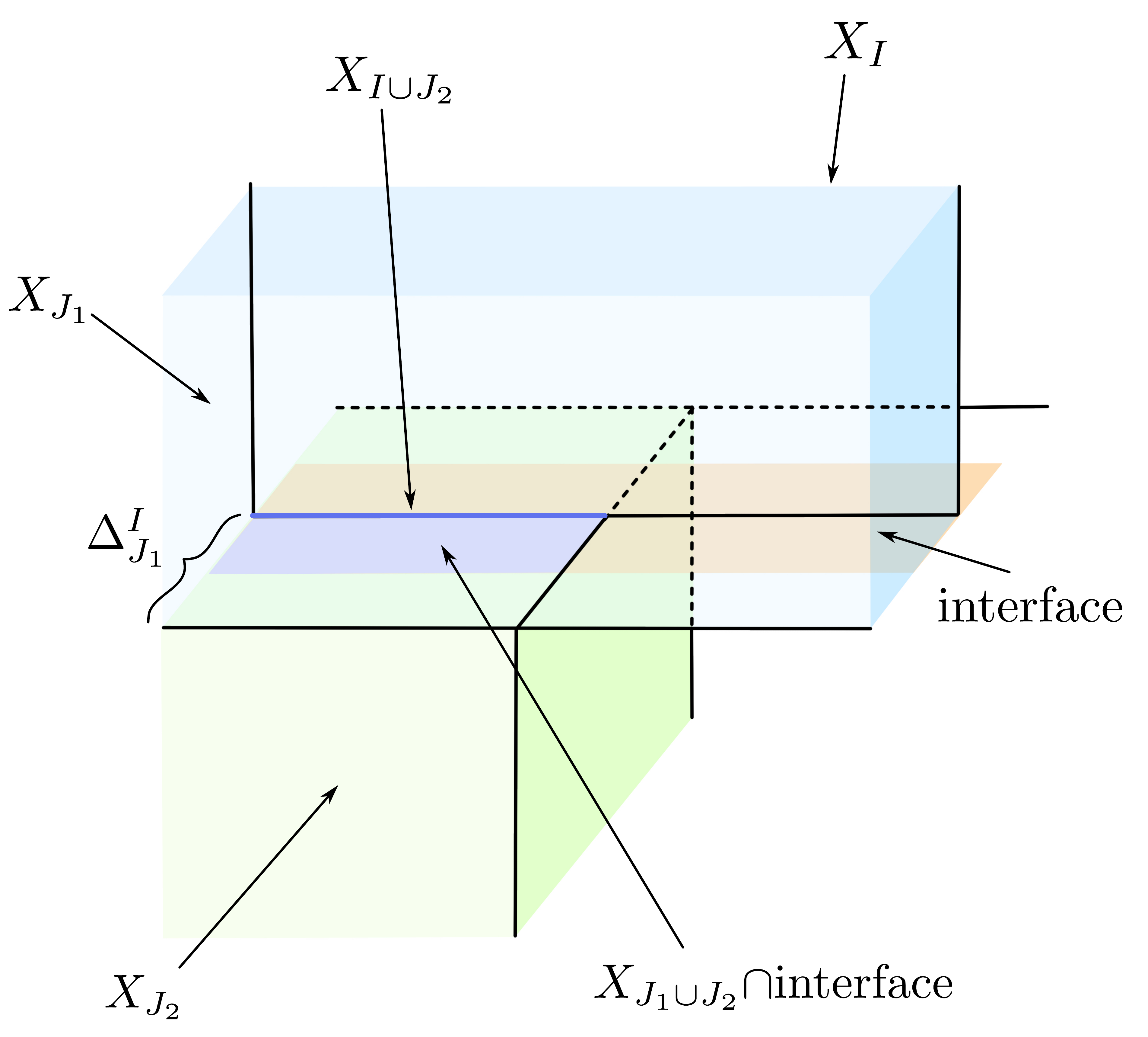}
        \caption{A neighborhood of $X_I$ in the multisection of $X$}
    \end{subfigure}
    \begin{subfigure}[t]{0.4\linewidth}
        \centering
        \includegraphics[width=\linewidth]{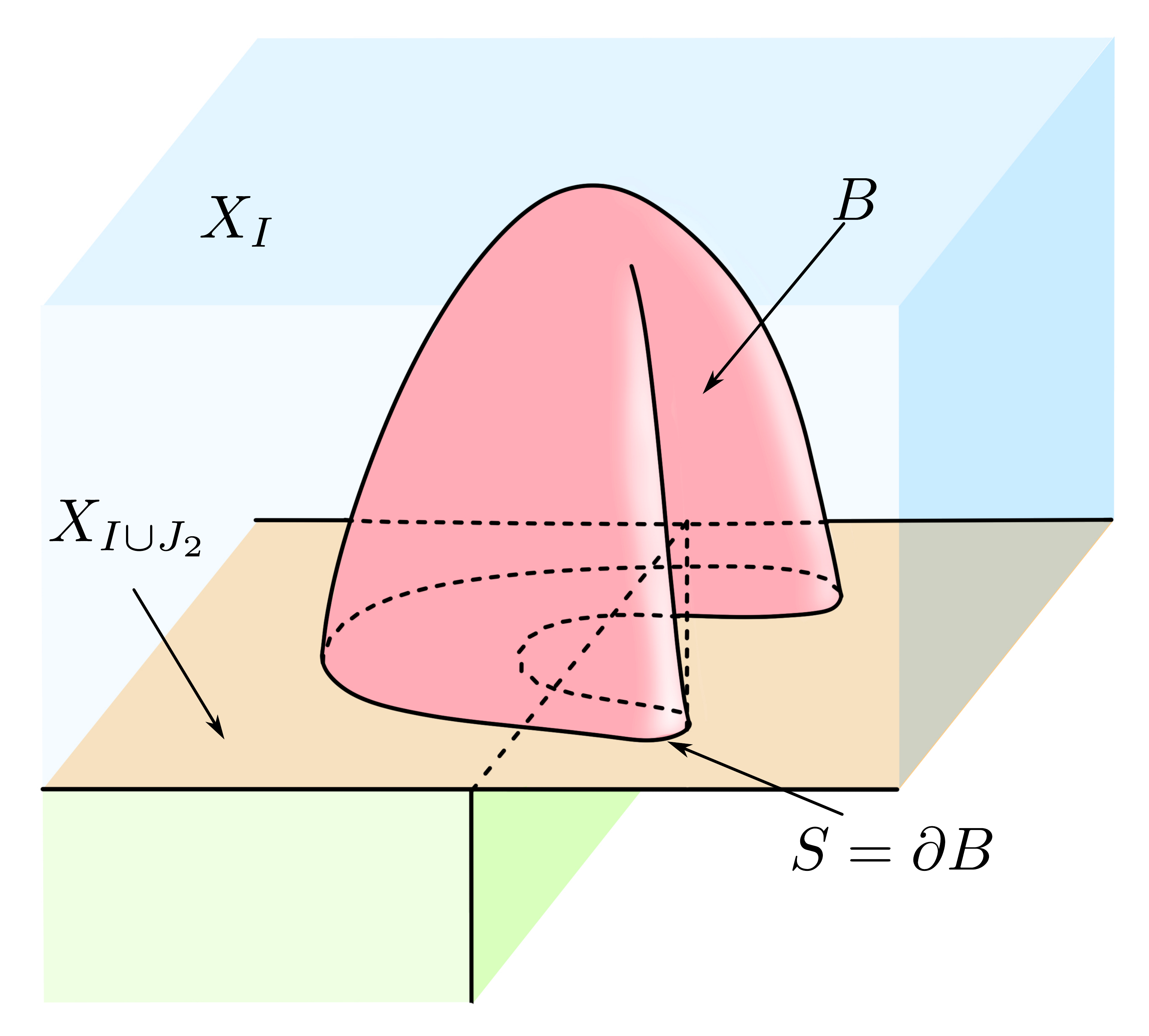}
        \caption{The sector $X_I$ and a collar neighborhood of $\partial X_I$ outside $X_I$, with the core of the handle $h$ drawn}
    \end{subfigure}
\caption{A schematic picture showing the relative position of the core $B$ of the $k$-handle $h$ in $X_I$, with attaching sphere $S\subset\partial X_I$. The thickened handle (not drawn) $\nu\cong D^{n-m}\times h$ is an $n$-dimensional $k$-handle attached to $\cup_{i\in I^c}X_i$, whose attaching region lies in the interface $\Delta^I\times\partial X_I$. Here, $J_1\subset I$, $J_2\subset I^c$ are index sets as in Case 3 of the argument.}
\label{fig:k-hd}
\end{figure}

By construction, $X_J=X_{J_1}\cap X_{J_2}$, $X_J'=X_{J_1}'\cap X_{J_2}'=\overline{X_{J_1}\backslash\nu}\cap(X_{J_2}\cup\nu_{J_2})$. Since the modification occurs in a tubular neighborhood of $X_I$, we examine the change within the standard model $V=\Delta^I\times X_I$. We have $$X_J\cap V=\Delta^I_{J_1}\times X_{I\cup J_2}.$$
Let $\Int_v(h)$ denote $h$ as an open handle (i.e. $\Int_v(h)=\Int(h)\cup\Int(\partial_{bot}h)$) and similarly $\Int_v(\nu)$. Let $\Int_v(\Delta^I_\bullet)$ denote the open cone on $\Delta^{I\backslash\bullet}$ with cone point $0\in\Delta^I$. Thus, $\Int_v(\nu)_{J_1}=\tfrac12\Int_v(\Delta^I_{J_1})\times\Int_v(h)$. Write $\partial_v\Delta^I_\bullet=\Delta^{I\backslash\bullet}$ and $\partial_vh=\partial h\backslash\Int(\partial_{bot}(h))$. Thus, $\partial_v\Delta^I_\bullet=\Delta^I_\bullet\backslash\Int_v(\Delta^I_\bullet)$ and $\partial_vh=h\backslash\Int_v(h)$. We have
\begin{align*}
X_J'\cap V&=\Big((\Delta^I_{J_1}\times X_I)\setminus\big(\tfrac12\Int_v(\Delta^I_{J_1})\times\Int_v(h)\big)\Big)\cap\big(X_{J_2}\cup(\tfrac12\Delta^I\times h_{J_2})\big)\\
&=\Big((\Delta^I_{J_1}\times X_I)\cap\big(X_{J_2}\cup(\tfrac12\Delta^I\times h_{J_2})\big)\Big)\setminus\Big(\big(\tfrac12\Int_v(\Delta^I_{J_1})\times\Int_v(h)\big)\cap\big(X_{J_2}\cup(\tfrac12\Delta^I\times h_{J_2})\big)\Big)\\
&=\big((\Delta^I_{J_1}\times X_{I\cup J_2})\cup(\tfrac12\Delta^I_{J_1}\times h_{J_2})\big)\setminus\big(\tfrac12\Int_v(\Delta^I_{J_1})\times\Int_v(h)_{J_2}\big).
\end{align*}
The last row is a deformation retract of $(\Delta^I_{J_1}\times X_{I\cup J_2})\cup(\tfrac12\Delta^I_{J_1}\times h_{J_2})$, although the two spaces are not diffeomorphic since the latter is not a manifold. Nevertheless, the attachment of $\tfrac12\Delta^I_{J_1}\times h_{J_2}$ to $\Delta^I_{J_1}\times X_{I\cup J_2}$ is an attachment of some $0,1$-handles by Lemma~\ref{lem:handle_color}, and the deformation retract flattens these handles while preserving their indices. We see that $X_J'\cap V$ is obtained from $(\Delta^I_{J_1}\times X_{I\cup J_2})\backslash(\tfrac12\Delta^I_{J_1}\times\Int(\partial_{bot}h)_{J_2})$ by attaching $0,1$-handles. Since (after corner smoothing) $(\Delta^I_{J_1}\times X_{I\cup J_2})\backslash(\tfrac12\Delta^I_{J_1}\times\Int(\partial_{bot}h)_{J_2})$ is diffeomorphic to $\Delta^I_{J_1}\times X_{I\cup J_2}=X_J\cap V$, we see that $X_J'$ is obtained from $X_J$ by attaching $0,1$-handles. Turning upside down, we conclude that $X_J'$ admits a (rel boundary) handle decomposition with the same number of handles of every index, except possibly for some additional handles with coindices $0,1$. This finishes the proof of Proposition~\ref{prop:modification}, and hence of Theorem~\ref{thm:multisection}.\qed

\begin{Rmk}
In the proof of Theorem~\ref{thm:multisection}, we have only used the existence of bridge position (Theorem~\ref{thm:bridge}) in a weak sense, namely that we needed that every sector $S_I\subset X_I$ is a disjoint union of disks, but not that they are $\partial$-parallel.
\end{Rmk}

\section{Definitions: General case}\label{sec:def_general}
In this section, we generalize the notions of multisections and bridge positions to the case of manifolds with corners. In practice, we often smooth corners to regard them simply as manifolds with boundary.

\subsection{Multisection}
In what follows, we will sometimes consider compact manifolds $M$ with corners, whose boundary $\partial M$ comes with a decomposition $\partial_-M\cup\partial_+M$ into two codimension-$0$ submanifolds, called the \textit{bottom boundary} and the \textit{top boundary} of $M$, respectively, glued along their common boundary $\partial(\partial_-M)=\partial(\partial_+M)$ which is assumed to be a union of strata in $M$ of codimension at least $2$. After smoothing corners except leaving $\partial(\partial_\pm M)$ as a codimension-$2$ corner, and thickening $\partial(\partial_\pm M)$ into a collar in $\partial M$, one can regard such a manifold $M$ as a cobordism from $\partial_-M$ to $\partial_+M$ rel boundary, and talk about handle decompositions of $M$ rel $\partial_-M$. See Figure~\ref{fig:rel_cob}.

\begin{figure}[htb]
\begin{center}
\begin{tikzpicture} 
\begin{scope}
 \draw (1,2) -- (0,1) -- (1,0) -- (3,1) -- (1,2) -- (1,0);
 \draw[dashed] (0,1) -- (3,1);
 \draw[white,line width=5pt,fill=white] (1.4,1.5) .. controls +(0.7,-0.3) and +(-0.8,0.5) .. (3,2.5) .. controls +(0.8,-0.5) and +(-0.6,0.4) .. (2.1,0.9);
 \draw[fill=blue,opacity=0.1] (0,1) -- (1,0) -- (3,1);
 \draw (1.4,1.5) .. controls +(0.7,-0.3) and +(-0.8,0.5) .. (3,2.5) .. controls +(0.8,-0.5) and +(-0.6,0.4) .. (2.1,0.9);
 \draw (2.5,1.9) arc (-90:20:0.3) (2.8,2.2) arc (80:200:0.2);
 \draw[blue!40] (0,0.2) node {$\partial_-M$};
 \draw (0.2,1.8) node {$\partial_+M$};
\end{scope}
\begin{scope} [xshift=4cm]
 \draw[->] (0,1) -- (2,1);
 \draw (1,1) node[above] {smooth} node[below] {corners};
\end{scope}
\begin{scope} [xshift=8cm,yshift=0.5cm]
 \draw (0,1) -- (0,0) arc (-130:-50:1.5) -- (1.93,1) arc (-50:-130:1.5) arc (130:50:1.5);
 \draw[fill=blue,opacity=0.1] (0,0) arc (-130:-50:1.5) arc (50:130:1.5);
 \draw[dashed] (0,0) arc (130:50:1.5);
 \draw[white,line width=5pt,fill=white] (0.5,1) .. controls +(0.5,0.2) and +(-0.8,0) .. (1,2.5) .. controls +(0.8,0) and +(-0.5,0.2) .. (1.5,1);
 \draw (0.5,1) .. controls +(0.5,0.2) and +(-0.8,0) .. (1,2.5) .. controls +(0.8,0) and +(-0.5,0.2) .. (1.5,1);
 \draw (0.95,1.7) arc (-55:55:0.3) (1,1.77) arc (240:120:0.2);
 \draw[blue!40] (0,-0.6) node {$\partial_-M$};
 \draw (0,1.8) node {$\partial_+M$} (3.2,0.5) node {$I\times\partial(\partial_{\pm}M)$};
\end{scope}
\end{tikzpicture}
\end{center}
    \caption{Regard a manifold $M$ with corners as a relative cobordism from $\partial_-M$ to $\partial_+M$.}
    \label{fig:rel_cob}
\end{figure}
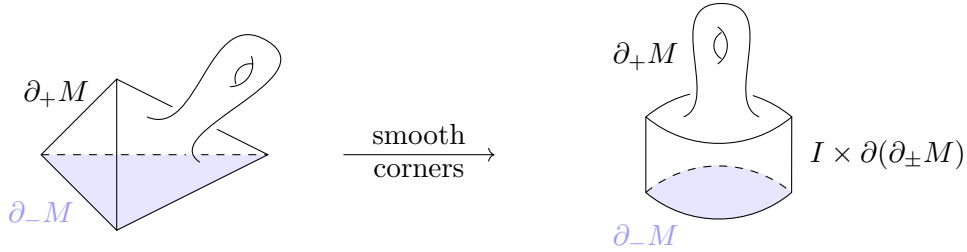

\begin{Def}[Multisection with corners]\label{def:multisection_general}
Let $n\ge2$ be an integer. An \textit{$(n-1)$-section} (or \textit{multisection}) of a compact $n$-manifold $X$ with corners is a decomposition $X=X_1\cup\cdots\cup X_{n-1}$, so that if we set $X_I:=\cap_{i\in I}X_i$ for $\emptyset\ne I\subset\{1,\cdots,n-1\}$, the following conditions hold:
\begin{enumerate}
\item Each $X_I\subset X$ is a submanifold with corners, whose codimension-$k$ stratum contains $$\bigcup_{J\supset I,|J|-|I|=k}\Int(X_J)$$ as a union of connected components. It is also assumed that $\partial X_I$ is in general position with respect to the corner stratification on $\partial X$.
\item Let $\partial_-X_I:=\cup_{J\supsetneq I}X_J$ and $\partial_+X_I:=\partial X_I\cap\partial X=\overline{\partial X_I\backslash\partial_-X_I}$. Then, after smoothing the corners, $X_I$ admits a handle decomposition rel $\partial_-X_I$ with handles of coindices $0$ and $1$.
\end{enumerate}
\end{Def}
\begin{Rmk}
If $X=X_1\cup\cdots\cup X_{n-1}$ is a multisection with corners, then $\Sigma:=X_{\{1,\cdots,n-1\}}\subset X$ is an embedded surface with corners.
\end{Rmk}

Let $\Delta^I$, $\Delta^I_J$ be as in Section~\ref{sbsec:def_multisection_closed}. If $X=X_1\cup\cdots\cup X_{n-1}$ is a multisection with corners, $\emptyset\ne I\subset\{1,\cdots,n-1\}$, then $\Delta^I$ is a local transverse model for $X_I$.

\begin{Ex}
If $X=X_1\cup\cdots\cup X_{n-1}$ is a multisection with corners, then for $I\ne\emptyset,\{1,\cdots,n-1\}$, after the appropriate corner-smoothing, $\partial_-X_I=\cup_{j\not\in I}X_{I\cup\{j\}}$ is a multisection with corners.
\end{Ex}

\begin{Ex}
If $X=X_1\cup\cdots\cup X_{n-1}$ is a multisection of a closed manifold in the sense of Definition~\ref{def:multisection}, then it is a multisection in the sense of Definition~\ref{def:multisection_general}, with $\partial X_I=\partial_-X_I$ for all $I$.

More generally, if $X$ is a multisected closed manifold and $S\subset X$ is a closed submanifold in bridge position, then $X\backslash\nu(S)$ with its induced decomposition from $X$ is a multisection with corners, since one can check that deleting a $\partial$-parallel disk in a sector affects the relative handle decomposition of the sector by adding a coindex-$1$ handle. This example will be further generalized in Lemma~\ref{lem:bridge_complement}.
\end{Ex}

\subsection{Bridge position}
\begin{Def}[Bridge position in a multisection with corners]\label{def:bridge_general}
Let $X=X_1\cup\cdots\cup X_{n-1}$ be a multisection of a manifold with corners. A properly embedded submanifold $S\subset X$ with corners is said to be in \textit{bridge position} if the following conditions hold for $S_I:=S\cap X_I$, $\partial_\pm S_I:=S_I\cap\partial_\pm X_I$, $\emptyset\ne I\subset\{1,\cdots,n-1\}$:
\begin{enumerate}
\item The submanifold $S$ is in general position with respect to the decomposition of $X$. More precisely, $S$ is transverse to each sector $X_I$ within each stratum of $X$.
\item For each $I\ne\{1,\cdots,n-1\}$, after smoothing the corners, $(X_I,S_I)$ is obtained from (a thickening of) $(\partial_-X_I,\partial_-S_I)$ by attaching ambient handles for $S_I$ with coindex $0$, and then handles for $X_I$ with coindices $0,1$.
\end{enumerate}
\end{Def}
Thus, every sector $S_I\subset X_I$ of a submanifold $S\subset X$ in bridge position consists of $\partial_-$-parallel disks as well as cylinders between $\partial_\pm X_I$.

The usefulness of Definitions~\ref{def:multisection_general} and~\ref{def:bridge_general} is reflected in the following lemma, which will allow us to prove a strengthening of Theorem~\ref{thm:bridge} inductively.

\begin{Lem}\label{lem:bridge_complement}
If $X$ is a multisection with corners and $S\subset X$ is a properly embedded submanifold with corners in bridge position, then $X\backslash\nu(S)$, with its induced decomposition from $X$, is a multisection with corners.
\end{Lem}
\begin{proof}
In each sector $S_I\subset X_I$, $I\ne\{1,\cdots,n-1\}$, every ambient handle for $S_I$ of coindex $i$ gives rise to a handle of $X_I\backslash\nu(S_I)$ of coindex $i+1$; see \cite[Proposition~6.2.1]{gompf20234}. The statement follows from the case $i=0$.
\end{proof}

\section{Bridge positions}\label{sec:bridge}
In this section, we prove the following generalization of Theorem~\ref{thm:bridge} in the case of manifolds with corners.

\begin{Thm}\label{thm:bridge_general}
If $X=X_1\cup\cdots\cup X_{n-1}$ is a multisection with corners and $S\subset X$ is a properly embedded submanifold with codimension at least $2$, whose boundary is in general position in $\partial X$ (that is, $S$ is transverse to each $X_I$ within each stratum of $X$ of positive codimension), then $S$ can be isotoped rel boundary to be in bridge position.
\end{Thm}

The strategy is analogous to that in Section~\ref{sec:multisection}, as follows: In Section~\ref{sbsec:complexity_sub}, we define the \textit{complexity} of a properly embedded submanifold $S\subset X$ in general position in a multisection with corners, valued in a well-ordered set, so that $S$ has complexity $0$ (the unique minimum) if and only if it is in bridge position in $X$. Then, in Sections~\ref{sbsec:modification_sub} and~\ref{sbsec:proof_decrease_sub}, we show that it is always possible to isotope $S$ rel boundary to decrease its complexity if the complexity is nonzero. This implies that $S$ can be isotoped rel boundary to be in bridge position in finitely many steps, proving Theorem~\ref{thm:bridge_general} and hence Theorem~\ref{thm:bridge} as a special case. We will induct on $n=\dim X$. The base case is $n=2$, where there is nothing to prove.

\subsection{Complexity of a submanifold}\label{sbsec:complexity_sub}
We first define the complexity of a cobordism rel boundary of pairs of manifolds, designed specifically with our problem in mind.

\begin{Def}
Let $(M,S)\colon(\partial_-M,\partial_-S)\to(\partial_+M,\partial_+S)$ be a cobordism rel boundary of pairs of manifolds with boundary, where $S$ has dimension $s$. The \textit{complexity} of $(M,S)$ is $$c(M,S):=\min\{(\#\text{ambient }(s-1)\text{-handles},\cdots,\#\text{ambient }0\text{-handles)}\}\in(\Z_{\ge0}\cup\{\infty\})^s,$$ where $(\Z_{\ge0}\cup\{\infty\})^s$ is equipped with the lexicographical order, and the minimum is taken over all handle decompositions of $(M,S)$ rel $(\partial_-M,\partial_-S)$ in which ambient handles for $S$ are attached before handles for $M$. If there is no such handle decomposition, the complexity is set to be $(\infty,\cdots,\infty)$. The \textit{complexity} of a pair $(M,S)$ of manifolds with corners where $\partial M$ comes with a decomposition $\partial M=\partial_+M\cup\partial_-M$ is defined to be the complexity of $(M,S)$ as a cobordism rel boundary between pairs of manifolds with boundary (illustrated on $M$ as in Figure~\ref{fig:rel_cob}).
\end{Def}

\begin{Def}\label{def:complexity_sub}
Let $X=X_1\cup\cdots\cup X_{n-1}$ be a multisection with corners and $S\subset X$ be a properly embedded $s$-dimensional submanifold with corners in general position, $s\le n-2$. The \textit{complexity} of $(X,S)$ is $$c(X,S)=c(X_1,\cdots,X_{n-1};S)=\left(\sum_{|I|=1}c(X_I,S_I),\cdots,\sum_{|I|=s}c(X_I,S_I)\right)\in\Z_{\ge0}^{s(s+1)/2}.$$
\end{Def}
Note that in Definition~\ref{def:complexity_sub}, all coordinates of $c(X_1,\cdots,X_{n-1};S)$ are finite. This is because in each sector $X_I$, by the assumption on codimension, one can assume that $S_I$ is disjoint from all handles of $X_I$ rel $\partial_-X_I$, which are of coindices $0,1$.

Theorem~\ref{thm:bridge_general} is implied by the following proposition.

\begin{Prop}\label{prop:modification_sub}
Let $(X,S)$ be a pair as in Definition~\ref{def:complexity_sub}. Let $c_r(\cdot)$ denote the $r$-th coordinate of $c(\cdot)$. If $r$ is the maximal index with $c_r(X_1,\cdots,X_{n-1};S)>0$, then there is some submanifold $S'$ in general position that is isotopic to $S$ rel boundary, such that $c_\ell(X_1,\cdots,X_{n-1};S')\le c_\ell(X_1,\cdots,X_{n-1};S)$ for all $\ell\le r$, with strict inequality when $\ell=r$.
\end{Prop}

\subsection{The complexity-decreasing isotopy}\label{sbsec:modification_sub}
In this section, we propose the appropriate isotopy from $S$ to $S'$ for proving Proposition~\ref{prop:modification_sub}. Fix handle decompositions on the sectors realizing the complexity $c(X_1,\cdots,X_{n-1};S)$. By assumption, for some $3\le m\le n$, $0\le t<s-n+m$, in some $m$-dimensional sector $X_I$, $S_I$ has an ambient $t$-handle, and no $m'$-dimensional sector has an ambient $t'$-handle if $t'<s-n+m'$ and $(m',t')<(m,t)$ in lexicographical order. In particular, $\partial_-S_I=S_I\cap\partial_-X_I$ is in bridge position in the multisection $\partial_-X_I$ (after the appropriate corner-smoothing). We isotope $S$ to remove one ambient $t$-handle in $S_I\subset X_I$, so that no sector $S_J\subset X_J$ with $|J|\le|I|$ gets more ambient handles with positive coindices.

To start, we fix some notation. Let $h$ be an ambient $t$-handle in $S_I\subset X_I$, $B=B(h)$ be its core, and $h\cong_\alpha D^{s-t-n+m}\times B$. Let $\nu=\nu(h)\cong_\beta D^{n-m}\times h$ be a closed neighborhood of $h$ in the codimension-$0$ submanifold $\cup_{i\in I}S_i$ in $S$. Here, the diffeomorphisms $\alpha,\beta$ are fixed once and for all.

Since there are no ambient $t'$-handles in $S_I\subset X_I$ with $t'<t$, by the description of ambient handle decompositions in \cite{gompf20234}, for example, we may assume $h\subset\{1/2\}\times\partial_-X_I$ lies flat in some level of a standard closed collar neighborhood $[0,1]\times\partial_-X_I$ of $\partial_-X_I$ in $X_I$, that $S_I$ is the product $[0,1/2)\times\partial_-S_I$ below the $1/2$-level and a product in between the $(1/2,1]$-levels, so that $h$ is the unique handle up to level $1$. This description can be propagated to a closed neighborhood of $X_I\subset\cup_{i\in I}X_i$ as follows. Let $\Delta^I\times X_I$ denote a closed neighborhood of $X_I$, with bottom boundary $\Delta^I\times\partial_-X_I$, called the \textit{interface}, and $[0,1]\times\Delta^I\times\partial_-X_I$ be a closed collar neighborhood of the interface. Then, inside $\cup_{i\in I}X_i$, $S$ is the product $[0,1/2)\times\Delta^I\times\partial_-S_I$ below the $1/2$-level and a product in between the $(1/2,1]$-levels, and $\nu(h)=\{1/2\}\times\tfrac12\Delta^I\times h$ is the unique handle of $S$ up to the level $1$, lying flat in the $1/2$-level. Here, by an abuse of notation, we think of $\partial_-S_I$, $h$, $B$ as all lying in $\partial_-X_I$.

By Lemma~\ref{lem:bridge_complement}, $Y:=\partial_-X_I\backslash\nu(\partial_-S_I)$ inherits the structure of a multisection with corners from that of $\partial_-X_I$. By a small perturbation, we may assume the attaching sphere $\partial B$ of $h$, thought of as a submanifold of $\partial Y$, is in general position with respect to the corner stratification. By the induction hypothesis, we may put $B\subset Y$ in bridge position by an isotopy rel boundary. This isotopy can be extended to $h$ and $\nu(h)$, keeping them flat in the $1/2$-level above the interface. From now on, we still think of $B$ as attached to $\partial_-S_I$ as opposed to $\partial(\nu(\partial_-S_I))$.

Finally, we extend the collar neighborhood $[0,1]\times\Delta^I\times\partial_-X_I$ of the interface to a two-sided closed collar neighborhood $P:=[-1,1]\times\Delta^I\times\partial_-X_I$ such that both the decomposition of $X$ and the submanifold $S$ are standard, namely $X_J\cap P=[-1,0]\times\Delta^I\times X_{I\cup J}$ for $\emptyset\ne J\subset I^c$, and $S\cap([-1,0]\times\Delta^I\times\partial_-X_I)=[-1,0]\times\Delta^I\times\partial_-S_I$. In particular, when viewing $S\cap P$ in $P$ from level $-1$ to level $1$, everything is standard except the $s$-dimensional ambient $t$-handle $\nu(h)$ which lies flat in the level $1/2$.

\textbf{The $(n,m,s,t)$-modification ($3\le m\le n$, $0\le t\le s-n+m-1$):}\smallskip\\
Push $\nu(h)=\{1/2\}\times\tfrac12\Delta^I\times D^{s-t-n+m}\times B\subset\{1/2\}\times\Delta^I\times\partial_-X_I$ vertically down across the interface $\Delta^I\times\partial_-X_I=\{0\}\times\Delta^I\times\partial_-X_I$, from the level $1/2$ to the level $-1/2$.

More explicitly, the $(n,m,s,t)$-modification replaces the $$(\{1/2\}\times\tfrac12\Delta^I\times D^{s-t-n+m}\times B)\cup([-1/2,1/2]\times\tfrac12\Delta^I\times D^{s-t-n+m}\times\partial B)\subset P$$ part of $S$ with $$(\{-1/2\}\times\tfrac12\Delta^I\times D^{s-t-n+m}\times B)\cup([-1/2,1/2]\times\partial(\tfrac12\Delta^I\times D^{s-t-n+m})\times B)\subset P.$$ Denote the new submanifold by $S'$.

\begin{figure}[ht]
    \centering
    \begin{subfigure}[t]{0.44\linewidth}
        \centering
        \includegraphics[width=\linewidth]{X.pdf}
        \caption{A neighborhood of $X_I$ in the multisection of $X$}
    \end{subfigure}
    \begin{subfigure}[t]{0.48\linewidth}
        \centering
        \includegraphics[width=\linewidth]{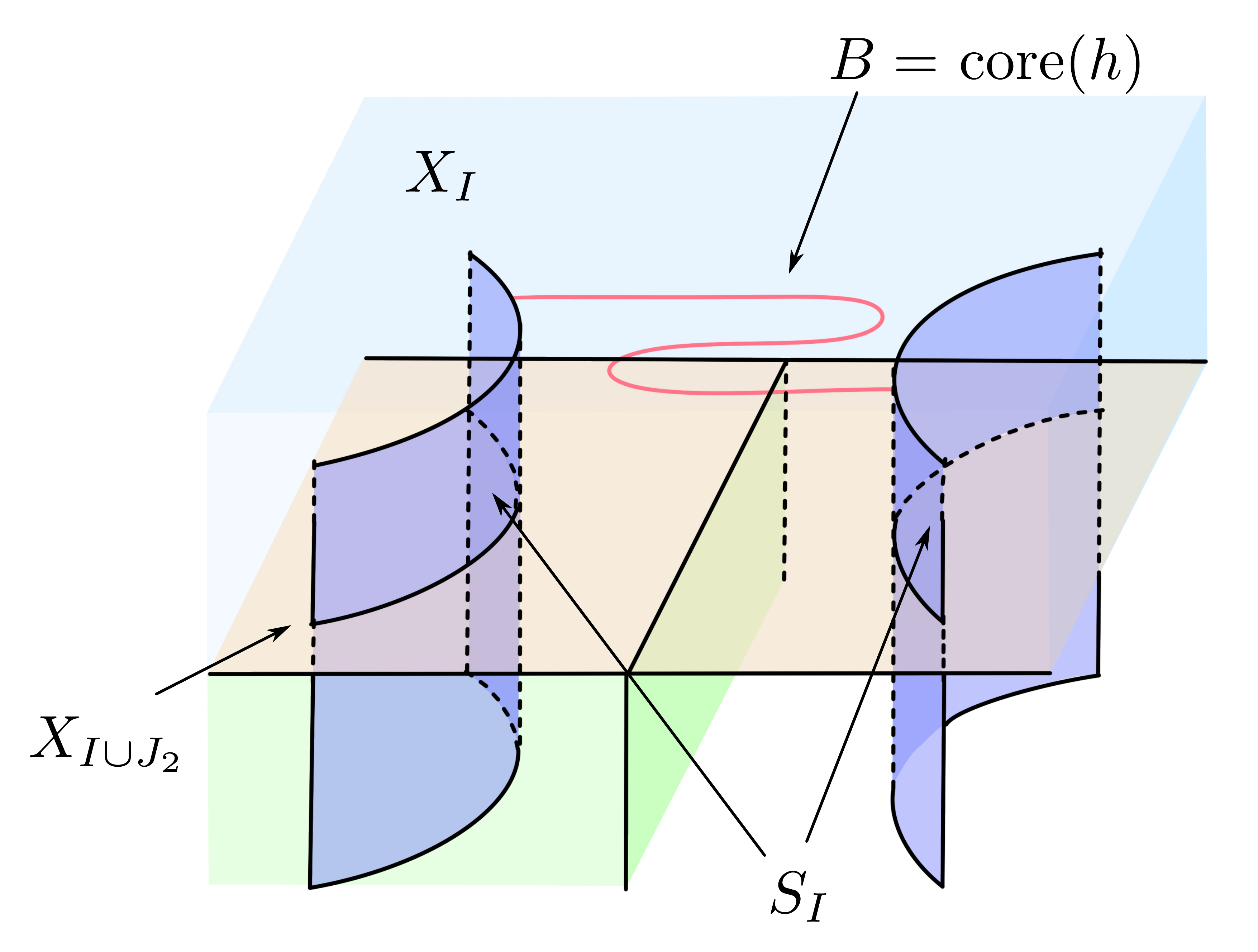}
        \caption{The sector $X_I$ and a collar neighborhood of $\partial X_I$ outside $X_I$, with the submanifold $S$ and the core of the ambient handle $h$ drawn}
    \end{subfigure}

    \caption{A schematic picture showing the relative position of the core $B$ of the ambient $t$-handle $h$ in $S_I\subset X_I$ lying at the $1/2$-level above the bottom boundary $\partial_-X_I$. The thickened handle (not drawn) $\nu\cong\tfrac12\Delta^I\times h$ is an $s$-dimensional ambient $t$-handle that shadows down onto the interface $\Delta^I\times\partial_-X_I$ via a product region $[0,1/2]\times\nu$. Here, $J_1\subset I$, $J_2\subset I^c$ are index sets as in Case 3 of the argument.}
    \label{fig:t-hd}
\end{figure}

\subsection{Proof of the complexity-decreasing property}\label{sbsec:proof_decrease_sub}
We prove that the proposed $(n,m,s,t)$-modification decreases the complexity of the decomposition. To do this, we analyze the change in the ambient handle decomposition in each sector $S_J\subset X_J$. It suffices to show that for every $|J|\le|I|$, $S_J'$ (admits an ambient handle decomposition that) has no more ambient $i$-handles than $S_J$ for $i\le\dim(S_J)-1$, and that $S_I'$ has fewer ambient $t$-handles than $S_I$.

\textbf{Case 0}: $J=I$.

By construction, after the isotopy, $S_I'$ is cylindrical in the collar $[0,1]\times\partial_-X_I$ and agrees with $S_I$ outside this region, effectively removing the $t$-handle $h$ from $S_I$. Thus, $S_I'$ has one fewer ambient $t$-handle than $S_I$ (and the same number of ambient handles of other indices).

\textbf{Case 1}: $J\subsetneq I$.

Inside $X_J$, the modification pushes the half-disk region $\tfrac12\Delta^I_J\times h$ in $S_J\subset X_J$ vertically down. One may see that $S_J'$ is isotopic to $S_J$ by an isotopy supported in a small neighborhood of the product region above the interface shadowed down from $\tfrac12\Delta^I_J\times h$, which does not intersect $\partial_+X_J$. To be explicit, we observe that $S_J$ and $S_J'$ agree outside the ball $$D_J:=[0,1/2]\times\tfrac12\Delta^I_J\times D^{s-t-n+m}\times B.$$ Both $S_J$ and $S_J'$ are disjoint from $\Int(D_J)$. Within the sphere $\partial D_J$, $S_J'$ differs from $S_J$ by replacing the ball $$B_{J,1}:=(\{1/2\}\times\tfrac12\Delta^I_J\times D^{s-t-n+m}\times B)\cup([0,1/2]\times\tfrac12\Delta^I_J\times D^{s-t-n+m}\times\partial B)$$ with the ball $$B_{J,2}:=[0,1/2]\times(\tfrac12\partial_v\Delta^I_J\times D^{s-t-n+m}\cup\tfrac12\Delta^I_J\times\partial D^{s-t-n+m})\times B,$$ where $\partial_v\Delta^I_J:=\partial\Delta^I\cap\partial\Delta^I_J$ is the external boundary of $\Delta^I_J$. The ball $D_J$ intersects the bottom boundary $\partial_-X_J$ in the ball $$B_{J,3}:=(\{0\}\times\tfrac12\Delta^I_J\cup[0,1/2]\times\tfrac12\partial_r\Delta^I_J)\times D^{s-t-n+m}\times B,$$ where $\partial_r\Delta^I_J:=\cup_{J\subsetneq K\subset I}\Delta^I_K$ is the radial/internal boundary of $\Delta^I_J$.

One may check that the decomposition $\partial D_J=B_{J,1}\cup B_{J,2}\cup B_{J,3}$ (after corner smoothing) is diffeomorphic to the standard decomposition of the $k$-sphere into three pieces for $k=s-|J|+1$ ($\ge2$). In particular, we can properly isotope (a push-in of) $B_{J,1}$ to (a push-in of) $B_{J,2}$ in $D_J$ rel $B_{J,1}\cap B_{J,2}$, giving the desired isotopy from $S_J$ to $S_J'$.

\textbf{Case 2}: $J\subset I^c$.

Write $B_J=B\cap X_{I\cup J}\subset X_{I\cup J}$ (here we view $B$ in $\partial_-X_I$). By assumption, it is a union of $\partial_-$-parallel disks together with some cylindrical region with its two ends on $\partial_-X_{I\cup J}$ and $\partial_-S_I\cap X_{I\cup J}$. The modification pushes the $\{1/2\}\times\tfrac12\Delta^I\times D^{s-t-n+m}\times B_J$ part of $\nu(h)$ across the interface into $X_J$. One may see that the cylindrical region does not change the handle decomposition of $S_J$, and that each $\partial_-$-parallel disk in $B_J$ contributes to a new coindex-$0$ ambient handle of $S_J'$. We check this as follows for the two cases individually. As a shorthand, let $B_1$ denote the disk $\tfrac12\Delta^I\times D^{s-t-n+m}$.

\textbf{2.1.} Let $B_0\subset B_J$ be a $\partial_-$-parallel disk. The corresponding part of $\nu(h)$ is $\{1/2\}\times B_1\times B_0$. After the push, locally $S_J$ changes from $\emptyset$ to $$(\{-1/2\}\times B_1\times B_0)\cup([-1/2,0]\times\partial B_1\times B_0),$$ which is parallel to the disk $(\{0\}\times B_1\times B_0)\cup([-1/2,0]\times B_1\times\partial B_0)\subset\partial_-X_J$ via the product region $[-1/2,0]\times B_1\times B_0$. Thus, the pushing introduces a new ambient handle in $S_J'$ of coindex $0$.

\textbf{2.2.} Let $[0,1]\times N\subset B_J$ be the cylindrical region, where $\partial_-([0,1]\times N):=\{0\}\times N\subset\partial_-X_{I\cup J}$ and $\partial_+([0,1]\times N):=[0,1]\times\partial N\cup\{1\}\times N\subset\partial_-S_I\cap X_{I\cup J}$. The corresponding part of $\nu(h)$ is $\{1/2\}\times B_1\times[0,1]\times N$. Locally near the region below $\{1/2\}\times B_1\times[0,1]\times N$, $S_J$ and $S_J'$ agree outside the product region $$Q_J:=[-1/2,0]\times B_1\times[0,1]\times N.$$ Both $S_J$ and $S_J'$ are disjoint from $\Int(Q_J)$. Within $\partial Q_J$, $S_J'$ differs from $S_J$ by replacing $$R_{J,1}:=[-1/2,0]\times B_1\times\partial_+([0,1]\times N)$$ with $$R_{J,2}:=(\{-1/2\}\times B_1\times[0,1]\times N)\cup([-1/2,0]\times\partial B_1\times[0,1]\times N).$$ The region $Q_J$ intersects the bottom boundary $\partial_-X_J$ in $$R_{J,3}:=(\{0\}\times B_1\times[0,1]\times N)\cup([-1/2,0]\times B_1\times\partial_-([0,1]\times N)).$$ Note that $\partial Q_J=R_{J,1}\cup R_{J,2}\cup R_{J,3}$ is a decomposition into codimension-$0$ submanifolds where the pieces have pairwise intersections in codimension-$1$ submanifolds and the triple intersection is a codimension-$2$ submanifold. It suffices to prove that (a push-in of) $R_{J,1}$ can be properly isotoped in $Q_J$ to (a push-in of) $R_{J,2}$ rel $R_{J,1}\cap R_{J,2}$, which would imply that $S_J'$ is locally obtained from $S_J$ by an isotopy. To see this, we note that the boundary sphere of the ball $B':=[-1/2,0]\times B_1\times[0,1]$ is standardly decomposed into the union of three balls $B_1':=[-1/2,0]\times B_1\times\{1\}$, $B_2':=(\{-1/2\}\times B_1\times[0,1])\cup([-1/2,0]\times\partial B_1\times[0,1])$, $B_3':=(\{0\}\times B_1\times[0,1])\cup([-1/2,0]\times B_1\times\{0\})$. We can rewrite $$Q_J=B'\times N,\ R_{J,1}=(B_1'\times N)\cup(B'\times\partial N),\ R_{J,2}=B_2'\times N,\ R_{J,3}=B_3'\times N.$$ It is then easy to observe that $R_{J,1}\cap R_{J,3}=((B_1'\cap B_3')\times N)\cup(B_3'\times\partial N)$ and $R_{J,2}\cap R_{J,3}=(B_2'\cap B_3')\times N$ are parallel in $R_{J,3}$. We may thus properly isotope $R_{J,1}$ to $R_{J,1}\cup R_{J,3}=((B_1'\cup B_3')\times N)\cup(B'\times\partial N)$ rel $R_{J,1}\cap R_{J,2}$ by dragging a part of its boundary. Then, we may properly isotope $R_{J,1}\cup R_{J,3}$ to $R_{J,2}$ in $Q_J$ rel their common boundary as they are parallel.

\textbf{Case 3}: $J=J_1\cup J_2$, where $\emptyset\ne J_1\subsetneq I$, $\emptyset\ne J_2\subset I^c$.

Locally in the interface, $X_J$ is the product $\{0\}\times\Delta^I_{J_1}\times X_{I\cup J_2}$, with $\partial_-X_J$ being $\{0\}\times(\partial_r\Delta^I_{J_1}\times X_{I\cup J_2}\cup\Delta^I_{J_1}\times\partial_-X_{I\cup J_2})$ (recall $\partial_r\Delta^I_\bullet$ and $\partial_v\Delta^I_\bullet$ are defined in Case 1). In what follows, we will omit the prefix ``$\{0\}\times$'' for submanifolds in the interface.

Write $B_{J_2}=B\cap X_{I\cup J_2}\subset X_{I\cup J_2}$. It is a union of $\partial_-$-parallel disks together with some cylindrical region with its two ends on $\partial_-X_{I\cup J_2}$ and $\partial_-S_I\cap X_{I\cup J_2}$, as in Case 2. The modification pushes the $\{1/2\}\times\tfrac12\Delta^I_{J_1}\times D^{s-t-n+m}\times B_{J_2}$ part of $\nu(h)$ across the interface, leaving the surgery of $S_J$ as the new submanifold in $X_J$. We check that the push of the cylindrical region does not change the handle decomposition of $S_J$, and that the push of each $\partial_-$-parallel disk contributes to a new coindex-$0$ ambient handle of $S_J'$, in a way analogous to Case 2.

\textbf{3.1.} Let $B_0\subset B_{J_2}$ be a $\partial_-$-parallel disk. The corresponding part of $\nu(h)$ is $\{1/2\}\times\tfrac12\Delta^I_{J_1}\times D^{s-t-n+m}\times B_0$. After the push, locally $S_J$ changes from $\emptyset$ to $$(\tfrac12\partial_v\Delta^I_{J_1}\times D^{s-t-n+m}\cup\tfrac12\Delta^I_{J_1}\times\partial D^{s-t-n+m})\times B_0,$$ which is parallel to the disk $$(\tfrac12\partial_r\Delta^I_{J_1}\times D^{s-t-n+m}\times B_0)\cup(\tfrac12\Delta^I_{J_1}\times D^{s-t-n+m}\times\partial B_0)\subset\partial_-X_J$$ via the product region $\tfrac12\Delta^I_{J_1}\times D^{s-t-n+m}\times B_0$. Thus, $S_J'\subset X_J$ locally has one more ambient handle of coindex $0$.

\textbf{3.2.} Let $[0,1]\times N\subset B_{J_2}$ be the cylindrical region, where $\partial_-([0,1]\times N):=\{0\}\times N\subset\partial_-X_{I\cup J_2}$ and $\partial_+([0,1]\times N):=([0,1]\times\partial N)\cup(\{1\}\times N)\subset\partial_-S_I\cap X_{I\cup J_2}$. The corresponding part of $\nu(h)$ is $\{1/2\}\times\tfrac12\Delta^I_{J_1}\times D^{s-t-n+m}\times[0,1]\times N$. Locally near the region on the interface below this part of the handle, $S_J$ and $S_J'$ agree outside the product region $$Q_J:=\tfrac12\Delta^I_{J_1}\times D^{s-t-n+m}\times[0,1]\times N.$$ Both $S_J$ and $S_J'$ are disjoint from $\Int(Q_J)$. Within $\partial Q_J$, $S_J'$ differs from $S_J$ by replacing $$R_{J,1}:=\tfrac12\Delta^I_{J_1}\times D^{s-t-n+m}\times\partial_+([0,1]\times N)$$ with $$R_{J,2}:=(\tfrac12\partial_v\Delta^I_{J_1}\times D^{s-t-n+m}\cup\tfrac12\Delta^I_{J_1}\times\partial D^{s-t-n+m})\times[0,1]\times N.$$ The region $Q_J$ intersects the bottom boundary $\partial_-X_J$ in $$R_{J,3}:=(\tfrac12\partial_r\Delta^I_{J_1}\times D^{s-t-n+m}\times[0,1]\times N)\cup(\tfrac12\Delta^I_{J_1}\times D^{s-t-n+m}\times\partial_-([0,1]\times N)).$$ We conclude that $R_{J,1}$ is properly isotopic to $R_{J,2}$ in $Q_J$ rel $R_{J,1}\cap R_{J,2}$ in exactly the same way as in \textbf{2.2}.

This finishes the induction step, proving Proposition~\ref{prop:modification_sub}, and hence Theorem~\ref{thm:bridge_general} and, in particular, Theorem~\ref{thm:bridge}.\qed

\subsection{PL case}
Since we work with handle decompositions, the results in this paper are also true in the PL category. Indeed, PL manifolds and pairs admit triangulations (in which the link of each vertex is a sphere), and one can construct a handle decomposition from a triangulation. The goal of this subsection is to check the existence of the precise handle decompositions we use to define the complexity in Section~\ref{sbsec:complexity_sub}: given a properly embedded submanifold $S$ in a $1$-handlebody $X$, there is a handle decomposition rel boundary with first ambient handles for $S$, and then handles of $X$ of coindices $0$ and $1$.

In this section, all manifolds and submanifolds are PL.

We work with a pair $(X,S)$ where $X$ is an $n$-dimensional manifold with boundary and $S$ is an $s$-dimensional properly embedded submanifold of $X$. We consider handle decompositions rel boundary with two types of handles:
\begin{itemize}
 \item standard $k$-handles of $X$: $D^k\times D^{n-k}$ glued along $S^{k-1}\times D^{n-k}$,
 \item ambient $k$-handles for $S$: $(D^k\times D^{s-k}\times \Delta^{n-s},D^k\times D^{s-k}\times\{*\})$ with $*\in\Int(\Delta^{n-s})$, glued along $\big((S^{k-1}\times D^{s-k}\times\Delta^{n-s})\cup(D^k\times D^{s-k}\times\partial_-\Delta^{n-s}),S^{k-1}\times D^{s-k}\times *\big)$, where $\Delta^n$ denotes the {\em standard $n$-dimensional half-disk}, that is an $n$-disk whose boundary is decomposed into two disks $\partial_-\Delta^n$
and $\partial_+\Delta^n$ which intersect precisely along their boundary; see Figure~\ref{fig:ambienthandle}. 
We define the {\em attaching disk} of such a handle as the $k$-disk $(D^k\times\{p\}\times\{0\})\cup(S^{k-1}\times\{p\}\times[0,1])$, where $p\in\Int(D^{s-k})$ and the interval $[0,1]$ is embedded in $\Delta^{n-s}$ so that $0$ is identified with $[0,1]\cap\partial\Delta^{n-s}\subset\partial_-\Delta^{n-s}$ and $1$ is identified with $*\in\Int(\Delta^{n-s})$.
\end{itemize}

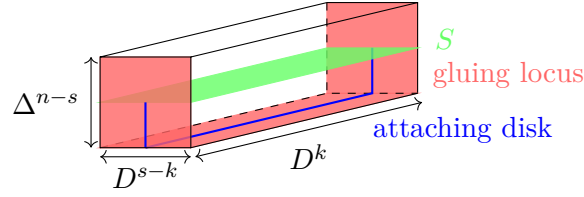
\begin{figure}[htb]
\begin{center}
\begin{tikzpicture} [scale=1.2]
 \draw[red!60,fill=red!60,opacity=0.8] (2.5,0.6) -- (3.5,0.6) -- (3.5,1.6) -- (2.5,1.6) -- (2.5,0.6);
 \draw (3.5,0.6) -- (3.5,1.6) -- (2.5,1.6);
 \draw[red!60,fill=red!60,opacity=0.8] (2.5,0.6) -- (3.5,0.6) -- (1,0) -- (0,0) -- (2.5,0.6);
 \draw[dashed] (2.5,1.6) -- (2.5,0.6) -- (3.5,0.6);
 \draw[blue,thick] (3,1.1) -- (3,0.6) -- (0.5,0);
 \draw[green!60,fill=green!60,opacity=0.8] (2.5,1.1) -- (3.5,1.1) -- (1,0.5) -- (0,0.5) -- (2.5,1.1);
 \draw[dashed] (0,0) -- (2.5,0.6);
 \draw[red!60,fill=red!60,opacity=0.8] (0,0) -- (1,0) -- (1,1) -- (0,1) -- (0,0);
 \draw[blue,thick] (0.5,0) -- (0.5,0.5);
 \draw (0,1) -- (2.5,1.6) (1,1) -- (3.5,1.6) (3.5,0.6) -- (1,0) -- (1,1) -- (0,1) -- (0,0) -- (1,0);
 \draw[red!60] (4.5,0.8) node {gluing locus};
 \draw[green!60] (3.8,1.2) node {$S$};
 \draw[blue] (4,0.2) node {attaching disk};
 \draw[<->] (1.05,-0.1) -- (3.55,0.5);
 \draw (2.3,-0.1) node {$D^k$};
 \draw[<->] (0,-0.1) -- (1,-0.1);
 \draw (0.5,-0.3) node {$D^{s-k}$};
 \draw[<->] (-0.1,0) -- (-0.1,1);
 \draw (-0.6,0.5) node {$\Delta^{n-s}$};
\end{tikzpicture}
\end{center}
\caption{An ambient handle. Here, $n=3$, $s=2$, and $k=1$.}
\label{fig:ambienthandle}
\end{figure}

\begin{Prop} \label{prop:handledec}
 Every pair $(X,S)$ admits a handle decomposition where the handles are ordered by increasing indices, and, for a given index, the standard handles are glued before the ambient handles. Further, handles of the same type and index can be glued simultaneously.
\end{Prop}
\begin{proof}
 Take a triangulation $\TT$ of $X$ that induces triangulations of $S$, $\partial X$, and $\partial S$. Let $\TT_2$ be the second barycentric subdivision of $\TT$. This provides a handle decomposition of $X$ as follows (see Figure~\ref{fig:triangulation}):
 \begin{itemize}
  \item the simplices that meet the boundary and their faces form a collar neighborhood of the pair $(\partial X,\partial S)$,
  \item the $k$-handles are the stars in $\TT_2$ of the centers of the $k$-simplices of $\TT$ not lying in $\partial X$.
 \end{itemize}
\begin{figure}[htb]
\begin{center}
\begin{tikzpicture} [xscale=0.87,scale=0.4,rotate=90]
\foreach \s in {-1,1} {
\foreach \t in {-1,1} {
\begin{scope} [xscale=\s,yscale=\t]
 \draw (0,12) -- (0,0);
 \draw (6,0) -- (0,4) -- (3,6);
 \draw (0,0) -- (3,2) -- (3,6) -- (0,8) (0,2) -- (6,0) -- (1.5,5) -- (0,12) (3,0) -- (0,4) -- (4.5,3) (0,4) -- (1.5,9);
 \draw[orange,fill=orange,opacity=0.5] (3,0) -- (6,0) -- (4.5,3) -- (3,3.33) -- (3,2) -- (2,1.33) -- (3,0) (0,12) -- (0,8) -- (1,7.33) -- (1.5,9) -- (0,12) (1.5,9) -- (1,7.33) -- (1.5,5) -- (3,3.33) -- (4.5,3) -- (1.5,9);
 \draw[purple,fill=purple,opacity=0.5] (0,0) -- (3,0) -- (2,1.33) -- (0,2) -- (0,0);
 \draw[blue,fill=blue,opacity=0.5] (0,2) -- (2,1.33) -- (3,2) -- (3,3.33) -- (1.5,5) -- (1,7.33) -- (0,8) -- (0,2);
 \draw[line width=2pt,green] (0,0) -- (6,0);
 \draw[thick] (0,12) -- (6,0);
 \foreach \x/\y in {0/0,0/4}
 \draw (\x,\y) node {$\bullet$};
\end{scope}}}
 \draw[green] (6,0) node[above right] {$S$};
 \draw[orange] (2,-18) node {collar of $(\partial X,\partial S)$};
 \draw[purple] (0,-20) node {1-handle of the pair $(X,S)$};
 \draw[blue] (-2,-17) node {2-handles of $X$};
\end{tikzpicture} \caption{From a triangulation to a handle decomposition.} \label{fig:triangulation}
\end{center}
\end{figure}
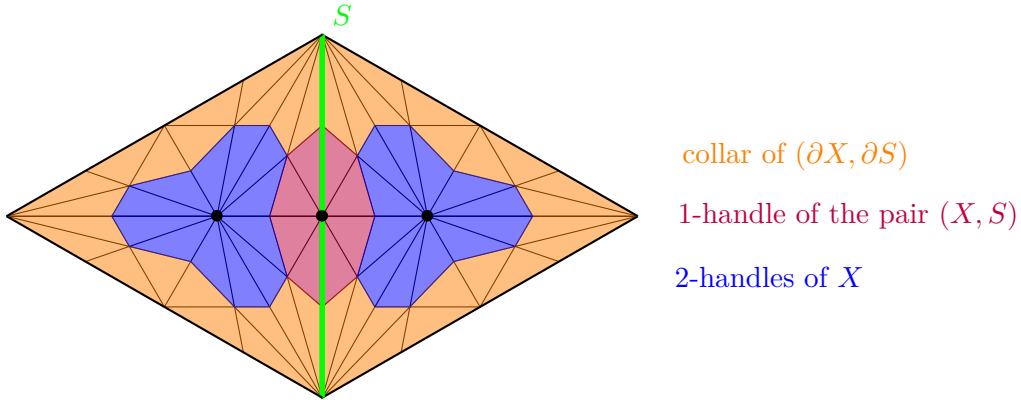
 This is the usual way to get a handle decomposition from a triangulation (see \cite[\S 6.8]{RS}). Note that the handles can be glued in order of increasing index. In this decomposition, some handles are disjoint from $S$, so that they are standard handles of~$X$, and the others contain a handle of $S$ of the same index. In the latter case, we further divide the handle into a standard handle of $X$ and an ambient handle for $S$. More precisely, let $D^k\times D^{s-k}\times D^{n-s}$ be a $k$-handle with gluing locus $S^{k-1}\times D^{s-k}\times D^{n-s}$, and containing a handle of $S$ given as $D^k\times D^{s-k}\times \{*\}$. Writing $D^{n-s}$ as the union of $\Delta_-^{n-s}$ and $\Delta_+^{n-s}$ glued along $\partial_+\Delta_-^{n-s}=\partial_-\Delta_+^{n-s}$, with $*\in\Int(\Delta_+^{n-s})$, we cut this handle into:
 \begin{itemize}
  \item first, a standard $k$-handle $D^k\times D^{s-k}\times \Delta_-^{n-s}$ with gluing locus $S^{k-1}\times D^{s-k}\times \Delta_-^{n-s}$,
  \item second, an ambient $k$-handle $D^k\times D^{s-k}\times \Delta_+^{n-s}$ with gluing locus $(S^{k-1}\times D^{s-k}\times \Delta_+^{n-s})\cup (D^k\times D^{s-k}\times \partial_-\Delta_+^{n-s})$, containing the $k$-handle of $S$.
 \end{itemize}
 See Figure~\ref{fig:handlepair}.
\end{proof}

\begin{figure}[htb]
\begin{center}
\begin{tikzpicture} [scale=1.2]
\begin{scope}
 \draw[red!60,fill=red!60,opacity=0.8] (2.5,0.6) -- (3.5,0.6) -- (3.5,1.6) -- (2.5,1.6) -- (2.5,0.6);
 \draw (3.5,0.6) -- (3.5,1.6) -- (2.5,1.6);
 \draw[dashed] (2.5,1.6) -- (2.5,0.6) -- (3.5,0.6);
 \draw[green!60,fill=green!60,opacity=0.8] (2.5,1.1) -- (3.5,1.1) -- (1,0.5) -- (0,0.5) -- (2.5,1.1);
 \draw[dashed] (0,0) -- (2.5,0.6);
 \draw[red!60,fill=red!60,opacity=0.8] (0,0) -- (1,0) -- (1,1) -- (0,1) -- (0,0);
 \draw (0,1) -- (2.5,1.6) (1,1) -- (3.5,1.6) (3.5,0.6) -- (1,0) -- (1,1) -- (0,1) -- (0,0) -- (1,0);
\end{scope}
 \draw[->,thick] (4,0.8) -- (5.5,0.8);
\begin{scope} [xshift=6cm]
\begin{scope} [yshift=0.3cm]
 \draw[red!60,fill=red!60,opacity=0.8] (2.5,0.9) -- (3.5,0.9) -- (3.5,1.6) -- (2.5,1.6) -- (2.5,0.9);
 \draw (3.5,0.9) -- (3.5,1.6) -- (2.5,1.6);
 \draw[red!60,fill=red!60,opacity=0.8] (2.5,0.9) -- (3.5,0.9) -- (1,0.3) -- (0,0.3) -- (2.5,0.9);
 \draw[dashed] (2.5,1.6) -- (2.5,0.9) -- (3.5,0.9);
 \draw[green!60,fill=green!60,opacity=0.8] (2.5,1.1) -- (3.5,1.1) -- (1,0.5) -- (0,0.5) -- (2.5,1.1);
 \draw[dashed] (0,0.3) -- (2.5,0.9);
 \draw[red!60,fill=red!60,opacity=0.8] (0,0.3) -- (1,0.3) -- (1,1) -- (0,1) -- (0,0.3);
 \draw (0,1) -- (2.5,1.6) (1,1) -- (3.5,1.6) (3.5,0.9) -- (1,0.3) -- (1,1) -- (0,1) -- (0,0.3) -- (1,0.3);
\end{scope}
\begin{scope} 
 \draw[red!60,fill=red!60,opacity=0.8] (2.5,0.6) -- (3.5,0.6) -- (3.5,0.9) -- (2.5,0.9) -- (2.5,0.6);
 \draw (3.5,0.6) -- (3.5,0.9) -- (2.5,0.9);
 \draw[dashed] (2.5,0.9) -- (2.5,0.6) -- (3.5,0.6);
 \draw[dashed] (0,0) -- (2.5,0.6);
 \draw[red!60,fill=red!60,opacity=0.8] (0,0) -- (1,0) -- (1,0.3) -- (0,0.3) -- (0,0);
 \draw (0,0.3) -- (2.5,0.9) (1,0.3) -- (3.5,0.9) (3.5,0.6) -- (1,0) -- (1,0.3) -- (0,0.3) -- (0,0) -- (1,0);
\end{scope}
\end{scope}
\end{tikzpicture}
\end{center}
\caption{Turning a handle pair into a standard handle followed by an ambient handle.}
\label{fig:handlepair}
\end{figure}
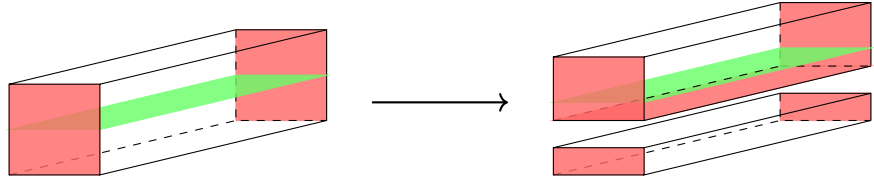

\begin{Prop}
 Let $X$ be an $n$-dimensional $1$-handlebody. Let $S\subset X$ be a properly embedded submanifold of dimension $s\leq n-2$. Then the pair $(X,S)$ admits a handle decomposition rel boundary with first ambient handles for $S$, and then handles of $X$ of coindices $0$ and $1$.
\end{Prop}
\begin{proof}
 We start with a handle decomposition $\HH$ of $(X,S)$ as provided by Proposition~\ref{prop:handledec}. If we forget $S$, then $\HH$ gives a standard handle decomposition of $X$. Now $X$ also admits a handle decomposition $\HH'$ with only handles of coindices $0,1$, which we can also assume to be ordered by increasing index. By uniqueness of handle decompositions, one can go from $\HH$ to $\HH'$ by a certain number of the following moves: sliding a handle over another handle of the same index, adding/removing a pair of handles in canceling position, isotoping the attaching spheres of the handles. We need to check that the ambient handles do not prevent us from performing these moves. For isotopies and additions of pairs of handles, it is straightforward. When sliding a $k$-handle over another, the sliding is defined along an arc living in the $(n-1)$-dimensional level where the $k$-handles are glued. The arc being $1$-dimensional, it generically avoids $S$, which has codimension at least~$2$. Finally, consider a $k$-handle $h_k$ and a $(k+1)$-handle $h_{k+1}$ in canceling position (if we only look at the handle decomposition of $X$), with $k\leq n-2$. The belt sphere of $h_k$ has dimension $n-k-1\geq1$. In the positive boundary of $h_k$, this belt sphere generically meets the attaching $k$-disks of the ambient $k$-handles along points, so that it is not covered by these ambient handles. Hence, up to isotoping the attaching sphere of $h_{k+1}$, we can assume that it meets the belt sphere of $h_k$ in a single point away from the ambient $k$-handles. It follows that the cancellation can be performed.
\end{proof}

\section{Trisecting \texorpdfstring{$5$}{5}-manifolds with boundary}\label{sec:trisection_5}
The purpose of this section is two-fold. First, we use our proof strategy of Theorem~\ref{thm:multisection} to give an alternative proof of the main theorem of Lambert-Cole--Miller \cite{lambert2021trisections} concerning trisecting $5$-manifolds with boundary. Second, we use this case as a low-dimensional example to make our proof idea of Theorem~\ref{thm:multisection} more concrete.

The notion of ``trisection'' in this section is not in the sense of Definition~\ref{def:multisection}, but rather in the spirit of the notion of multisections studied by Rubinstein--Tillmann \cite{rubinstein2020multisections}; see Section~\ref{sbsec:generalization} for a discussion of various generalizations of multisections, and in particular Example~\ref{ex:rubinstein_tillmann}(1) for the case relevant to the discussion here.

For the purpose of this section, a \textit{trisection} of a compact $5$-manifold $M$ with boundary is a decomposition $M=M_1\cup M_2\cup M_3$ into $5$-manifolds with corners, such that if we put $M_{ij}=M_i\cap M_j$, $M_{123}=M_1\cap M_2\cap M_3$, and $\partial_{bot}(\bullet)=\partial(\bullet)\cap\partial M$, then
\begin{enumerate}
\item $\partial M_1$ has codimension-$0$ stratum $\Int(M_{12})\cup\Int(M_{13})\cup\Int(\partial_{bot}M_1)$, codimension-$1$ stratum $\Int(M_{123})\cup\Int(\partial_{bot}M_{12})\cup\Int(\partial_{bot}M_{13})$, and codimension-$2$ stratum $\partial_{bot}M_{123}$. Similarly for $\partial M_2$ and $\partial M_3$.
\item After smoothing the corners, each $M_i$ can be built out of $0,1$-handles, each $M_{ij}$ can be built out of $0,1,2$-handles, and $M_{123}$ can be built out of $0,1,2,3$-handles.
\item $\partial M=\partial_{bot}M_1\cup\partial_{bot}M_2\cup\partial_{bot}M_3$ is a trisection in the sense of Definition~\ref{def:multisection}.
\end{enumerate}
See Figure~\ref{fig:trisection_5} for an illustration.
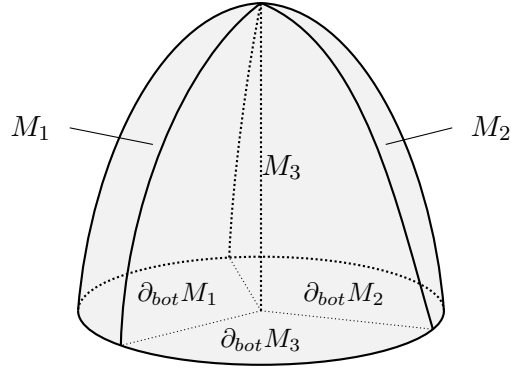
\begin{figure}
\begin{tikzpicture}[
    scale=1.1,
    front/.style={thick},
    back/.style={thick,densely dotted},
    radial/.style={densely dotted},
    leader/.style={thin},
]

% Parameters
\def\a{2.2}   % horizontal radius of bottom ellipse
\def\b{0.65}  % vertical radius of bottom ellipse

% Main coordinates
\coordinate (C) at (0,0);
\coordinate (S) at (0,3.7);
\coordinate (L) at (-\a,0);
\coordinate (R) at (\a,0);

% Three equally spaced boundary points on the bottom ellipse
\coordinate (P1) at ({\a*cos(100)},{\b*sin(100)});
\coordinate (P2) at ({\a*cos(220)},{\b*sin(220)});
\coordinate (P3) at ({\a*cos(340)},{\b*sin(340)});

% Fill for the half-ellipsoid
\fill[gray!10]
    (L) .. controls (-2.05,1.9) and (-1.15,3.65) .. (S)
        .. controls (1.15,3.65) and (2.05,1.9) .. (R)
    arc[start angle=0,end angle=-180,x radius=\a,y radius=\b]
    -- cycle;

% Back half of the bottom ellipse
\draw[back] (R) arc[start angle=0,end angle=180,x radius=\a,y radius=\b];

% Dome outline
\draw[front] (L) .. controls (-2.05,1.9) and (-1.15,3.65) .. (S);
\draw[front] (S) .. controls (1.15,3.65) and (2.05,1.9) .. (R);

% Front half of the bottom ellipse
\draw[front] (R) arc[start angle=0,end angle=-180,x radius=\a,y radius=\b];

% Dotted radial lines in the bottom disk
\draw[radial] (C) -- (P1);
\draw[radial] (C) -- (P2);
\draw[radial] (C) -- (P3);

% Dotted axis from summit to center of bottom disk
\draw[back] (C) -- (S);

% Curved surface lines from boundary points to summit
\draw[back]  (P1) .. controls (-0.35,1.55) and (-0.15,3.00) .. (S);
\draw[front] (P2) .. controls (-1.65,1.20) and (-0.95,3.00) .. (S);
\draw[front] (P3) .. controls ( 1.65,1.20) and ( 1.25,3.20) .. (S);

% Labels
\node[left]  at (-2.4,2.2) {$M_1$};
\draw[leader] (-2.3,2.2) -- (-1.3,2.0);
\node[right] at (2.4,2.2) {$M_2$};
\draw[leader] (2.3,2.2) -- (1.5,2.0);
\node[below] at (0.25,2) {$M_3$};
\node at (-1,0.2) {\small$\partial_{bot}M_1$};
\node at (1,0.2) {\small$\partial_{bot}M_2$};
\node at (0,-0.35) {\small$\partial_{bot}M_3$};
\end{tikzpicture}
\caption{A trisection of a $5$-manifold with boundary.}
\label{fig:trisection_5}
\end{figure}

The following is a slight generalization of Lambert-Cole--Miller \cite[Theorem~1.4]{lambert2021trisections}, removing the hypotheses on the orientation and the number of boundary components.
\begin{Thm}
Every trisection of the boundary of a smooth compact $5$-manifold can be extended to a trisection of the $5$-manifold.
\end{Thm}
The idea is to extend the trisection arbitrarily to a decomposition of the $5$-manifold into $3$ pieces, then perform modifications to remove handles in the sectors with undesired indices. Instead of mimicking closely the proof of Theorem~\ref{thm:multisection}, we make use of some low-dimensional features to modify some of the steps, allowing a simpler modification procedure.

\begin{proof}
Let $M$ be a compact $5$-manifold whose boundary comes with a trisection $\partial M=\partial_1M\cup\partial_2M\cup\partial_3M$. We construct a trisection $M=M_1\cup M_2\cup M_3$ so that $\partial_{bot}M_i=\partial_iM$.

Without loss of generality, assume $M$ is connected. By deleting a $5$-ball and equipping its boundary $4$-sphere with the standard genus-$0$ trisection if necessary, we may reduce to the case when $M$ has nonempty boundary.

We first extend the boundary trisection arbitrarily to a decomposition $M=M_1\cup M_2\cup M_3$ satisfying the genericity condition (1) of a trisection. We may pick this decomposition so that in addition,
\begin{enumerate}[(i)]
\item every component of $M_i$ meets $\partial_{bot}M_i$ nontrivially and every component of $M_{ij}$ meets $\partial_{bot}M_{ij}$ nontrivially;
\item every component of $M_{ij}$ meets $M_{123}$ nontrivially.
\end{enumerate}
For example, one can take $M_1,M_2$ to be standard closed collar neighborhoods of $\partial_1M,\partial_2M$, and $M_3$ to be the closure of $M\backslash(M_1\cup M_2)$. The conditions (i)(ii) will be preserved throughout the modification procedure.

From now on, we turn all sectors upside down, and talk about handle decompositions of them rel boundary; so, the indices of the handles are reversed. By assumption (i), each $M_i$ has handles with possible indices $1,2,3,4,5$, and each $M_{ij}$ has handles with possible indices $1,2,3,4$. We need to remove $1,2,3$-handles in each $M_i$ and $1$-handles in each $M_{ij}$.

\textbf{1. $1$-handles in $M_{23}$} (similarly $M_{12},M_{13}$): Let $h\subset M_{23}$ be a ($4$-dimensional) $1$-handle. Its attaching region is the tubular neighborhood of two points on $\partial M_{23}=M_{123}\cup_{\partial_{bot}M_{123}}\partial_{bot}M_{23}$. Since $\partial M$ is trisected, we know $\partial_{bot}M_{23}$ has no closed component. Thus, we may isotope the attaching region of $h$ to lie completely in $M_{123}$. Let $\nu(h)\subset M_2\cup M_3$ be a closed tubular neighborhood of the handle $h$. Now, the interior modification $$M_1':=M_1\cup\nu(h),\ M_2':=\overline{M_2\backslash\nu(h)},\ M_3':=\overline{M_3\backslash\nu(h)}$$ removes the $1$-handle $h$ from $M_{23}$, adds a new $4$-handle (the upside-down copy of $\nu(h)$) to $M_1$, adds a new $3$-handle to each of $M_{12}$ and $M_{13}$, does not affect the handle decompositions of $M_2$ and $M_3$, and preserves the conditions (i)(ii) for the decomposition.

\textbf{2. $1$-handles in $M_3$} (similarly $M_1,M_2$): Let $h\subset M_3$ be a $1$-handle. Its attaching region is the tubular neighborhood of two points on $\partial M_3=M_{13}\cup M_{23}\cup\partial_{bot}M_3$. As before, we may isotope the attaching region of $h$ to lie in $M_{13}\cup M_{23}$. By assumption (ii), we may further isotope the attaching region $D$ to the tubular neighborhood of two points in $\Int(M_{123})$. The attaching region $D$ admits a decomposition $D=D_1\cup D_2$ induced by $M_{13}\cup M_{23}$, where each $D_i$ is the disjoint union of two half $4$-balls. We extend this decomposition to a decomposition $h=h_1\cup h_2$ of $h$ into two half $1$-handles that intersect each other in a $4$-dimensional $1$-handle $h_{12}$. Now, the interior modification 
\begin{equation}\label{eq:interior_modification}
M_1':=M_1\cup h_1,\ M_2':=M_2\cup h_2,\ M_3':=\overline{M_3\backslash h}
\end{equation}
removes the $1$-handle $h$ from $M_3$, adds a $4$-handle to each of $M_1$ and $M_2$, and preserves the conditions (i)(ii) for the decomposition.

Before each iteration of the following two modifications, we perform the first modification to remove $1$-handles in all $M_{ij}$.

\textbf{3. $2$-handles in $M_3$} (similarly $M_1,M_2$): Let $h\subset M_3$ be a $2$-handle. Its attaching region is the tubular neighborhood of a circle in $\partial M_3=M_{13}\cup M_{23}\cup\partial_{bot}M_3$. Since $\partial_{bot}M_3$ is a $1$-handlebody, we may push the attaching region of $h$ off $\partial_{bot}M_3$. Since $\partial_{bot}M_{13}$ is obtained from $\partial_{bot}M_{123}$ by attaching $2,3$-handles and $M_{13}$ is obtained from $\partial M_{13}=M_{123}\cup\partial_{bot}M_{13}$ by attaching $2,3,4$-handles, we see that $M_{13}$ (and similarly $M_{23}$) is obtained from $M_{123}$ by attaching $4$-dimensional $2,3,4$-handles. Therefore, we may isotope the attaching region of $h$ to lie in a collar neighborhood of $M_{123}$ in $M_{13}\cup M_{23}$. For dimension reasons, we may then assume that the attaching circle $S$ of $h$ is the union of two arcs $S_1,S_2$ in $M_{13},M_{23}$, respectively. We extend this decomposition $S=S_1\cup S_2$ to a decomposition $B=B_1\cup B_2$ of the core disk $B$ of $h$ into two half disks, and further extend this, by a thickening, to a decomposition $h=h_1\cup h_2$ of the whole handle into two half $2$-handles that intersect each other in a $4$-dimensional $1$-handle $h_{12}$. Now, the interior modification \eqref{eq:interior_modification} removes the $2$-handle $h$ from $M_3$, does not affect the handle decomposition of $M_1$ and $M_2$ (as the half $2$-handles can be deformation retracted into $M_1$ and $M_2$), and preserves the conditions (i)(ii).

\textbf{4. $3$-handles in $M_3$} (similarly $M_1,M_2$): Let $h\subset M_3$ be a $3$-handle. Its attaching region is the tubular neighborhood of a $2$-sphere $S$ in $\partial M_3=M_{13}\cup M_{23}\cup\partial_{bot}M_3$. As before, we may isotope $S$ to lie in $M_{13}\cup M_{23}$. Since $M_{13}$ is obtained from $M_{123}$ by attaching $2,3,4$-handles, we may assume that $S\cap M_{13}$ misses the $3,4$-handles and is a nonempty union of cores of $2$-handles in $M_{13}$ rel $M_{123}$. Here, the nonempty assumption can be achieved by adding canceling $2,3$-handles in $M_{13}$ rel $M_{123}$, if necessary. Thus, the decomposition $M_{13}\cup M_{23}$ induces a decomposition $S=S_1\cup S_2$ where $S_1$ is a nonempty disjoint union of disks and $S_2$ is a planar surface with nonempty boundary. We extend this decomposition to a decomposition $B=B_1\cup B_2$ of the core $3$-ball of $h$ by defining $B_2$ to be a closed collar neighborhood of $S_2$ in $B$ and $B_1=\overline{B\backslash B_2}$ (this is the \textit{$1$-cone-off decomposition} used in the proof of Theorem~\ref{thm:multisection}; see Section~\ref{sbsec:modification}), and further extend this, by a thickening, to a decomposition $h=h_1\cup h_2$ of the whole handle $h$. By construction, $h_1$ is the tubular neighborhood of a star-shaped graph in $M_3$ with the leaves attached to $M_{13}$, and $h_2$ is a collar neighborhood of $h\cap M_{23}$ in $M_3$. Now, the interior modification \eqref{eq:interior_modification} removes the $3$-handle $h$ from $M_3$, adds some number of $4$-handles to $M_1$, does not change the handle decomposition of $M_2$, and preserves the conditions (i)(ii).

The interior modifications \textbf{1}, \textbf{2}, \textbf{3}, and \textbf{4} can be applied iteratively to change $M_1\cup M_2\cup M_3$ to a trisection of $M$ extending the given trisection on $\partial M$.
\end{proof}

\section{Remarks and future directions} \label{sec:future}
\subsection{Uniqueness of multisections and bridge positions}
In \cite{aribi2023multisections}, the authors introduced stabilization moves for multisected manifolds and asked whether two multisections of a given smooth manifold are isotopic after stabilizations. This is known to hold in dimensions $3$ (Reidemeister--Singer) and $4$ (\cite{gay2016trisecting}). The number of moves grows with the dimension: for an $n$-manifold, there are $2^{n-2}-1$ stabilization moves. We recall briefly how this is defined and compare with the $(n,m,k)$-modifications that are used in the proof of Theorem~\ref{thm:multisection}.

We say that a half-disk $\Delta$ in a manifold with boundary $X$ is \emph{standard} if it is transverse to $\del X$ and $\Delta\cap \del X=\del_- \Delta$. Note that $\del_+\Delta$ is then a boundary parallel disk. Let $I\subset\{1,\dots,n-1\}$, $I\neq \emptyset, \{1,\dots,n-1\}$ and consider a standard half-disk $\Delta \subset X_I$ of codimension $1$ such that for all $I\subset J\subsetneq \{1,\dots,n-1\}$, $\Delta\cap X_J$ is also a standard half-disk in $X_J$. Such a half-disk is unique up to a multisection-preserving isotopy if $X$ is connected. 
It can be interpreted as a canceling pair of handles of coindices $2$ and $1$ in $X_I$ (relative to $\del X_I$). The coindex-$2$ handle has core $\del_+\Delta$ and we may apply the $(n,m,m-2)$-modification from Section~\ref{sec:multisection} (with $m=\dim X_I=n-|I|+1$) to add this handle to the opposite sectors. This is valid since $\del_-\Delta$ is in (a very simple form of) bridge position in $\del X_I$ and the coindex-$1$ handle left in $X_I$ is allowed in a multisection (the genus of $X_I$ increases by one). The cone-off procedure used to define the $(n,m,m-2)$-modification takes a very simple form in this case: we just transport the decomposition of $\del_- \Delta\subset \del X_I$ to $\del_+ \Delta$ via a trivialization $\Delta\cong \del_-\Delta \times [0,1]$ (with some corners smoothed). This is equivalent to our cone-off process with any choice of preferred index in $I^c$ (in our setup above $I=\{m-1,\dots,n-1\}$ and we choose the preferred index $1$ to decompose the core disk in the ``$1$-cone-off decomposition''). In addition to the considerations in Section~\ref{sbsec:proof_decrease}, one may also check that sectors $X_J$ with $J\supsetneq I$ change by adding some handles of coindex $1$. Yet another possibility is to add the core of this handle to a single opposite sector (this is how it is written in \cite{aribi2023multisections}). These variations all result in isotopic multisections, which we call the \emph{$I$-stabilization} of the given multisection. The stabilization can be done in an arbitrarily small neighborhood of a point of the central surface $X_{\{1,\cdots,n-1\}}$. Another observation from \cite{aribi2023multisections} is that $I$-stabilization is equivalent to $I^c$-stabilization, so this leaves indeed $2^{n-2}-1$ genuinely different stabilization moves.

\begin{Que}
Given two multisections of a closed manifold $X$, do they become isotopic after some number of stabilization moves?    
\end{Que}

We have established the existence of multisections in any dimension in this paper jointly with the existence of bridge positions for submanifolds
of codimension at least $2$. It is tempting to follow the same pattern when approaching the uniqueness problem. We thus now give a brief discussion
of perturbation moves for submanifolds in bridge position in a multisected manifold and formulate a similar uniqueness question.

First define a $d$-dimensional quarter of a disk to be $Q=\{(x_1,\dots,x_d)\in \R^d, x_1^2+\dots+x_d^2\leq 1, x_1\geq 0, x_2 \geq 0\}$ and denote the subsets $\del^v Q=\{x_1=0\}$, $\del_- Q=\{x_2=0\}$ and $\del_+ Q=\{x_1^2+\dots+x_d^2=1\}$. Given a manifold with boundary $X$ and a properly embedded submanifold $S$, we say that a quarter of a disk $Q\subset X$ is \emph{standard} relative to $S$ if it is transverse to $\del X$ and satisfies $Q\cap \del X=\del_- Q$, $Q\cap S=\del_+ Q$. Let $X$ be a multisected $n$-manifold and $S\subset X$ an $s$-dimensional submanifold with $s\leq n-2$. We now describe the $I$-perturbation of $S$ for a fixed $\emptyset\ne I\subsetneq \{1,\dots,n-1\}$.
We pick an $(n-s-1)$-codimensional quarter of a disk $Q\subset X_I$ such that for any $I\subset J\subsetneq\{1,\dots,n-1\}$, $Q\cap X_J$ is standard in $X_J$ relative to $S_J$.
This can be understood as a canceling pair of ambient handles for $S_I$ of coindices $0$ and $1$; see Figure~\ref{fig:perturbation}. The product region shadowed by the core of the handle of coindex $1$ is the half-disk $\del^v Q$ and we can push $S$ along this half-disk as in the $(n,m,s,t)$-modification from Section~\ref{sec:bridge} with $t=s-n+m-1$. This is valid since $\del^v Q\cap \del_- Q$ is (in a very simple form) in bridge position in $\del X_I\backslash\nu(\partial S_I)$ and the coindex $0$ ambient handle left in $S_I$ is allowed for a bridge position. The submanifold $S'$ obtained from this pushing procedure is called the \emph{$I$-perturbation} of $S$.

\begin{figure}[htb]
\begin{center}
\begin{tikzpicture} [yscale=0.5]
\begin{scope}
 \draw (1.6,-1) -- (2,1);
 \draw[white,fill=green!30,opacity=0.8] (2,2) -- (0,0) -- (5,0) -- (7,2);
 \draw[blue,fill=blue!50] (2,1) .. controls +(0.2,1) and +(-0.5,0) .. (3,3.5) .. controls +(0.5,0) and +(-0.5,0) .. (4,2.5) -- (4,1);
 \draw (2,1) .. controls +(0.2,1) and +(-0.5,0) .. (3,3.5) .. controls +(0.5,0) and +(-0.5,0) .. (4,2.5) .. controls +(0.5,0) and +(-1,-1) .. (6,5);
 \draw[blue] (3,2) node {$Q$};
 \draw[blue] (3,4) node {$\partial_+Q$};
 \draw[blue] (3,0.5) node {$\partial_-Q$};
 \draw[blue] (4.5,1.6) node {$\partial^vQ$};
\end{scope}
 \draw[->,thick] (7.5,1) -- (8.5,1);
\begin{scope} [xshift=9cm]
 \draw (1.6,-1) -- (2,1);
 \draw (3.9,1) .. controls +(0,-1) and +(0,-1) .. (4.1,1);
 \draw[white,fill=green!30,opacity=0.8] (2,2) -- (0,0) -- (5,0) -- (7,2);
 \draw (2,1) .. controls +(0.2,1) and +(-0.5,0) .. (3,3.5) .. controls +(0.5,0) and +(-0.45,0) .. (3.8,2.5) .. controls +(0.1,0) and +(0,1) .. (3.9,1) (4.1,1) .. controls +(0,1) and +(-0.1,0) .. (4.2,2.5) .. controls +(0.45,0) and +(-1,-1) .. (6,5);
\end{scope}
\end{tikzpicture} \caption{An $I$-perturbation. Here, $n=3$, $s=1$ and $|I|=1$.} \label{fig:perturbation}
\end{center}
\end{figure}
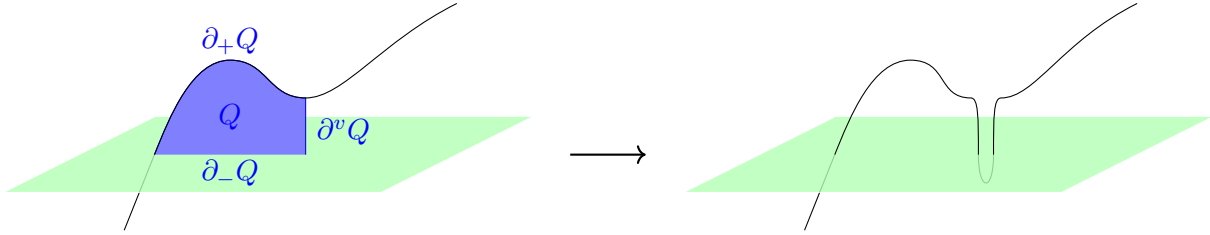

If $S$ is connected and of codimension $2$, the $I$-perturbation appears to be independent of choices up to bridge isotopy and equivalent to the $I^c$-perturbation (similarly as in \cite{aribi2023multisections}). So, we obtain $2^{n-2}-1$ different perturbation moves. For $n=3$, there is a single perturbation move and it is the standard one for links in Heegaard splittings. For $n=4$, the three perturbation moves have been introduced by Meier and Zupan in \cite{meier2017bridge}.

\begin{Que}
Let $S_0, S_1$ be closed submanifolds of codimension at least $2$ in bridge position in a closed multisected manifold $X$.  If $S_0$ and $S_1$ are isotopic, are they bridge isotopic after a certain number of perturbation moves?
\end{Que}
This was established for bridge surfaces in $S^4$ by Meier and Zupan \cite{meier2017bridge} and extended to arbitrary trisected $4$-manifolds by Hughes--Kim--Miller \cite{HKM2020}. This result on its own is potentially a step towards the uniqueness of quadrisections in dimension $5$, which is unknown at present.

\subsection{Generalizations of multisections}\label{sbsec:generalization}
Multisections provide a compact way to decompose a closed manifold into ``simple'' pieces, such that the intersections of the simple pieces are also ``simple.'' The purpose of this section is to make this description precise and discuss natural generalizations.

\subsubsection{Definitions and motivating questions}
\begin{Def}\label{def:generalized_multisections}
Let $n,k\ge0$ be integers. Let $\vec a=(a_1,\cdots,a_k)\in(\Z_{\ge0}\cup\{-\infty\})^k$ be a $k$-tuple. An \textit{$\vec a$-multisection} of a closed $n$-manifold $X$ is a decomposition $X=X_1\cup\cdots\cup X_k$, so that if we set $X_I:=\cap_{i\in I}X_i$ for $\emptyset\ne I\subset\{1,\cdots,k\}$, the following conditions hold:
\begin{enumerate}[(1)]
\item Each $X_I\subset X$ is a submanifold with corners, whose codimension-$\ell$ stratum is
$$\bigcup_{J\supset I,|J|-|I|=\ell}\Int(X_J).$$
\item After smoothing the corners, $X_I$ is a compact manifold with boundary built out of handles with indices at most $a_{|I|}$.
\end{enumerate}
\end{Def}
Thus, if $k=n-1$ and $\vec a=(1,\cdots,1,2)$, an $\vec a$-multisection is the same as a multisection in the sense of Definition~\ref{def:multisection}. In general, it is understood that $a_{|I|}=-\infty$ means that $X_I=\emptyset$.

We would like to understand when an $\vec a$-multisection exists for $n$-manifolds. For convenience, we make the following definition.

\begin{Def}
A $k$-tuple $\vec a\in(\Z_{\ge0}\cup\{-\infty\})^k$ is (smoothly) \textit{$n$-admissible} if every closed $n$-manifold admits an $\vec a$-multisection.
\end{Def}

Thus, our main theorem, Theorem~\ref{thm:multisection}, says that $(1,\cdots,1,2)\in(\Z_{\ge0}\cup\{-\infty\})^{n-1}$ is $n$-admissible.

\begin{Rmk}\label{rmk:var_adm_tuples}
If one considers PL manifolds, we obtain the notion of \textit{PL $n$-admissible tuples}. If one considers ($G$-)homology manifolds, we obtain the notion of \textit{($G$-)(co)homologically $n$-admissible tuples}, where the condition on the highest index of handles should be replaced by the highest nonvanishing degree of (co)homology. Since there are no topological handle decompositions in dimension $4$, we do not see an appropriate notion of topologically $n$-admissible tuples for $n\ge4$.
\end{Rmk}

Define a partial order $\le$ on the set of tuples $T:=\sqcup_{k\ge0}(\Z_{\ge0}\cup\{-\infty\})^k$ by requiring $(a_1,\cdots,a_k)\le(b_1,\cdots,b_\ell)$ if and only if $k\le\ell$ and $a_i\le b_i$ for $i\le k$. Then, if $\vec a\le\vec b$ and $\vec a$ is $n$-admissible, then so is~$\vec b$. An $n$-admissible tuple is \textit{minimal} if it is a minimal element in the subset $T_n\subset T$ of $n$-admissible tuples, with respect to the restriction of the partial order $\le$.

\begin{Que}\label{que:adm_tuples}
Classify all (smoothly, PL, (co)homologically) $n$-admissible tuples. Equivalently, write down the set of all minimal $n$-admissible tuples.
\end{Que}
Since we are mostly concerned with smooth manifolds, we simply write ``admissible tuples'' for ``smoothly admissible tuples.'' In this case, we shall see that the set of minimal $n$-admissible tuples is finite.

One can restrict to the subclass of nondecreasing $k$-tuples, namely tuples $(a_1,\cdots,a_k)$ with $a_1\le a_2\le\cdots\le a_k$. An $n$-admissible tuple is \textit{tight} if it is minimal and nondecreasing. We will see in Section~\ref{sbsbsec:adm_tuples_construction} that, in a certain sense, interesting admissible tuples turn out to be the tight ones. The special role of tight admissible tuples can also be seen in the obstruction theorem in Section~\ref{sbsbsec:adm_tuples_obstruction}. Finally, it appears that tight admissible tuples are much more tractable than minimal ones; see Section~\ref{sbsbsec:ex_lowdim}.

The following is a subquestion of Question~\ref{que:adm_tuples}.

\begin{Que}\label{que:tight_adm_tuples}
Write down the set of all tight $n$-admissible tuples.
\end{Que}

We end this section with an observation.

\begin{Obs}\label{obs:adm_tuples}
For a $k$-tuple $\vec a$, let $\vec a^{-,0}$ be the $k$-tuple defined as follows:
\begin{itemize}
\item if $i>n+1$ or if $a_j=-\infty$ for some $j\le i$, then $a^{-,0}_i=-\infty$;
\item otherwise, $a^{-,0}_i=\min(a_i,n-i+1)$.
\end{itemize}
Let $\vec a^-$ be the $k$-tuple defined as follows:
\begin{itemize}
\item if $a^{-,0}_i<n-i+1$ or if $i$ is the maximal index with $a^{-,0}_i=n-i+1$, then $a^-_i=a^{-,0}_i$;
\item otherwise, $a^-_i=a^{-,0}_i-1=n-i$.
\end{itemize}

We claim that $\vec a$ is $n$-admissible if and only if $\vec a^-$ is. Since $\vec a^-\le\vec a^{-,0}\le\vec a$, one direction is immediate. For the other direction, observe, for an $\vec a$-multisection $X=X_1\cup\cdots\cup X_k$, that
\begin{enumerate}[(i)]
\item $a_j=-\infty$ means that all $j$-fold intersections of an $\vec a$-multisection are empty, which implies all $i$-fold intersections are empty for $i\ge j$;
\item by condition (1) in Definition~\ref{def:generalized_multisections}, every $X_I$ is an $(n-|I|+1)$-manifold with corners, so it can be built out of handles of indices at most $n-|I|+1$.
\end{enumerate}
Thus every $\vec a$-multisection can be modified into an $\vec a^{-,0}$-multisection. Further,
\begin{enumerate}[(i)]
\setcounter{enumi}{2}
\item if there is a closed component $M$ of $X_I$ and $j>|I|$ with $a^{-,0}_j=n-j+1$, we may pick a ball $B\subset M$, subdivide it into $\ell=j-|I|$ pieces $B=B_1\cup\cdots\cup B_\ell$ in the standard way (obtained by thickening the standard decomposition of the $(\ell-1)$-simplex into $\ell$ pieces). Then we reassign a neighborhood $\nu(B)$ of $B$ to sectors $X_{j_1},\cdots,X_{j_\ell}$ for some distinct $j_1,\cdots,j_\ell\not\in I$ according to the decomposition of $B$. The new decomposition is still an $\vec a^{-,0}$-multisection, but the closed component $M\subset X_I$ is effectively punctured.
\end{enumerate}
Thus every $\vec a^{-,0}$-multisection can be modified into an $\vec a^-$-multisection.
\end{Obs}

\subsubsection{Constructions}\label{sbsbsec:adm_tuples_construction}
In this subsection, we present various natural constructions of admissible tuples. As we will show in Proposition~\ref{prop:minimal}, most of our examples turn out to be minimal admissible tuples.

\begin{Ex}\label{ex:n}
For any $n\ge0$, $(n)$ is a tight $n$-admissible $1$-tuple.
\end{Ex}
\begin{Ex}\label{ex:bisect}
For any $n>0$, $(\lfloor n/2\rfloor,n-1)$ is a tight $n$-admissible $2$-tuple. Admissibility can be seen by choosing a handle decomposition and dividing it into two halves, and minimality is an easy exercise.
\end{Ex}
\begin{Ex}\label{ex:0000}
For any $n\ge0$, by Moussard \cite[Lemma~3.1]{moussard2025multisections}, $(0,\cdots,0)$ is an $n$-admissible $(n+1)$-tuple. We shall see that this admissible tuple is tight. Consequently, it is the unique tight $n$-admissible tuple with length greater than $n$.
\end{Ex}
\begin{Ex}\label{ex:multisection}
By Theorem~\ref{thm:multisection}, for $n\ge2$, $(1,\cdots,1,2)$ is an $n$-admissible $(n-1)$-tuple; we shall see that it is tight.
\end{Ex}
\begin{Ex}\label{ex:111}
Suppose $X=X_1\cup\cdots\cup X_{n-1}$ is a multisection. We take a closed $n$-ball $X_n':=B$ near a point on the multisection surface $X_{\{1,\cdots,n-1\}}$. Let $X_i':=X_i\backslash\Int(B)$ for $i<n$. Then $X=X_1'\cup\cdots\cup X_n'$ is a $(1,\cdots,1)$-multisection of $X$, proving that $(1,\cdots,1)$ is an $n$-admissible $n$-tuple. We shall see that it is tight.
\end{Ex}
\begin{Ex}\label{ex:rubinstein_tillmann}
Rubinstein--Tillmann \cite{rubinstein2020multisections} introduced a different notion of multisections, and showed that every closed PL manifold admits a PL multisection in their sense. In our terminology, their main result \cite[Theorem~1.3]{rubinstein2020multisections} amounts to the following:
\begin{enumerate}
\item If $n=2k-1$ is odd, then $(1,2,\cdots,k)$ is a PL $n$-admissible $k$-tuple;
\item If $n=2k-2\ge4$ is even, then $(1,2,\cdots,k-2,k-2,k-1)$ is a PL $n$-admissible $k$-tuple.
\end{enumerate}
We shall later see that the PL admissible tuples in (1) are all minimal and hence tight, but it is not clear whether the tuples in (2) are minimal when $n\ge8$.

In dimensions at most $6$, every PL manifold has a unique compatible smooth structure up to isotopy (see \cite{hirsch1964smoothings}), so we obtain a tight (smoothly) $5$-admissible tuple $(1,2,3)$ and a tight (smoothly) $6$-admissible tuple $(1,2,2,3)$ that are not covered by the previous examples.

The notions of multisections defined by Rubinstein--Tillmann \cite{rubinstein2020multisections} and by Ben Aribi--Courte--Golla--Moussard \cite{aribi2023multisections} (Definition~\ref{def:multisection}) are both higher-dimensional generalizations of Heegaard splittings and trisections \cite{gay2016trisecting}, but they differ starting in dimension $5$. The latter notion has the advantage that the manifold can be recovered up to PL homeomorphisms (thus up to diffeomorphisms in dimension $n\le6$) by $n-1$ sets of attaching curves on the multisection surface that determine the $(n-1)$ $3$-dimensional sectors, at least when the manifold is orientable (see \cite[Section~3]{aribi2023multisections}). This could lead to a (a priori non-efficient) way to enumerate PL $n$-manifolds with increasing complexity.
\end{Ex}

\begin{Ex}\label{ex:lax}
If $\vec a=(a_1,\cdots,a_k)$ is an $n$-admissible $k$-tuple, and $a_k=n-k+1$, then for any $k<\ell\le n+1$, $\vec a^{(\ell)}:=(a_1,\cdots,a_{k-1},n-k,0,\cdots,0,n-\ell+1)$ is an $n$-admissible $\ell$-tuple. This is because $\vec b:=(a_1,\cdots,a_k,0,\cdots,0,n-\ell+1)$ is an $n$-admissible tuple and $\vec a^{(\ell)}\ge\vec b^-$ (see Observation~\ref{obs:adm_tuples}).

For instance, for any $3\le k\le n$, $(n-1,0,\cdots,0,n-k+1)$ is an $n$-admissible $k$-tuple. More explicitly, one can find an $(n-1,0,\cdots,0,n-k+1)$-multisection of a connected $n$-manifold $X$ by decomposing a small ball into $k-1$ sectors in the standard way and assigning the complement of the ball to a $k$-th sector. Since all interesting topology of $X$ is contained in the last piece, the decomposition is not particularly interesting. Although we shall see that $(n-1,0,\cdots,0,n-k+1)$ is minimal, heuristically, its nontightness reflects the boring feature of the decomposition.
\end{Ex}

\subsubsection{A cohomological obstruction}\label{sbsbsec:adm_tuples_obstruction}
We establish a cohomological obstruction to the $n$-admissibility of tuples by exploiting the cup product on the Mayer--Vietoris (\v Cech) spectral sequence associated to an $\vec a$-multisection.

For an element $\vec a\in T=\sqcup_{k\ge0}(\Z_{\ge0}\sqcup\{-\infty\})^k$, let $\vec a^{\ge0}\in T$ denote the element obtained from $\vec a$ by deleting all entries beyond (and including) the first $-\infty$ component in $\vec a$. For example, $(2,0,-\infty,3,-\infty)^{\ge0}=(2,0)$ and $(1,0,4)^{\ge0}=(1,0,4)$.

\begin{Thm}\label{thm:obstruction}
Suppose the $n$-torus $T^n$ admits a cohomological $\vec a$-multisection, where $\vec a^{\ge0}=(a_1,\cdots,a_k)$. Then, there exists a map $p\colon\{0,\cdots,n\}\to\{0,\cdots,k-1\}$ satisfying the following conditions:
\begin{enumerate}[(i)]
\item $p(i)\le i$ for all $i$.
\item $p(i)+p(j)\le p(i+j)$ for all $i,j$ with $i+j\le n$.
\item If $p(i)+p(j)=p(i+j)$, then $a_{r+1}\ge i-p(i)$ for all $p(i)\le r\le p(i)+p(j)$.
\end{enumerate}
\end{Thm}
Here, a cohomological $\vec a$-multisection is in the sense of Remark~\ref{rmk:var_adm_tuples}, namely a decomposition in which every $i$-fold intersection has trivial cohomology beyond dimension $a_i$. The function $p$ in Theorem~\ref{thm:obstruction} will be called a \textit{support function} of the $k$-tuple $\vec a$.

In the case of nondecreasing tuples, the obstruction simplifies significantly.
\begin{Cor}\label{cor:obstruction}
Suppose the $n$-torus $T^n$ admits a cohomological $\vec a$-multisection for some nondecreasing tuple $\vec a=(a_1,\cdots,a_k)$. Then, there exists a map $p\colon\{0,\cdots,n\}\to\{0,\cdots,k-1\}$ satisfying the following conditions:
\begin{enumerate}[(i)]
\item $0\le i-p(i)\le a_{p(i)+1}$ for all $i$.
\item $p(i)+p(j)\le p(i+j)$ for all $i,j$ with $i+j\le n$.
\end{enumerate}
\end{Cor}
\begin{proof}
By (i) of Theorem~\ref{thm:obstruction}, $p(0)=0$. Thus $i-p(i)\le a_{p(i)+1}$ by condition (iii) of Theorem~\ref{thm:obstruction} for $j=0$.
\end{proof}
Corollary~\ref{cor:obstruction} holds for all $\vec a$ with $\vec a=\vec a^{\ge0}$ regardless of whether $\vec a$ is nondecreasing. The point here is that when $\vec a$ is nondecreasing, the reduction from Theorem~\ref{thm:obstruction} to Corollary~\ref{cor:obstruction} is faithful.

\begin{proof}[Proof of Theorem~\ref{thm:obstruction}]
By Observation~\ref{obs:adm_tuples}, every cohomological $\vec a$-multisection can be replaced by a cohomological $\vec a^-$-multisection. Thus we may assume $\vec a=\vec a^-$, with some length $k'\ge k$. Let $T^n=X_1\cup\cdots\cup X_{k'}$ be a cohomological $\vec a$-multisection. The cohomological Mayer--Vietoris spectral sequence associated to the decomposition has $E_1$-page given by $E_1^{p,q}=\oplus_{|K|=p+1}H^q(X_K)$, supported in $0\le p<k$, $0\le q\le a_{p+1}$. The spectral sequence converges to $H^*(T^n)$, equipping it with a filtration $H^*(T^n)=F^0H^*(T^n)\supset F^1H^*(T^n)\supset\cdots\supset F^{k-1}H^*(T^n)\supset F^kH^*(T^n)=0$.

Fix a basis $u_1,\cdots,u_n\in H^1(T^n)\cong\Z^n$. As shorthand, write $u_I=\cup_{i\in I}u_i\in H^{|I|}(T^n)$ (arranged in increasing indices, say) for any $I\subset\{1,\cdots,n\}$. Then every $u_I$ is nonzero, thus has some filtration degree $p'(I)=\max\{r\colon u_I\in F^rH^{|I|}(T^n)\}\in\{0,\cdots,k-1\}$. Thus, the image of $u_I$ in $E_\infty^{p'(I),|I|-p'(I)}=F^{p'(I)}H^{|I|}(T^n)/F^{p'(I)+1}H^{|I|}(T^n)$ is nonzero, which implies that $E_1^{p'(I),|I|-p'(I)}\ne0$, whence
\begin{enumerate}[(i')]
\item $p'(I)\le|I|$.
\end{enumerate}
Since the cup product on $H^*(T^n)$ respects the filtration in the sense that $F^iH^*(T^n)\otimes F^jH^*(T^n)$ maps to $F^{i+j}H^*(T^n)$, we deduce that 
\begin{enumerate}[(i')]
\setcounter{enumi}{1}
\item $p'(I)+p'(J)\le p'(I\sqcup J)$ for disjoint $I,J\subset\{1,\cdots,n\}$.
\end{enumerate}
If the equality $p'(I)+p'(J)=p'(I\sqcup J)$ in (ii') is achieved, pick classes $x_I\in E_1^{p'(I),|I|-p'(I)}$, $x_J\in E_1^{p'(J),|J|-p'(J)}$ that survive in the $E_\infty$-page which represent $u_I,u_J$, respectively. Then $x_I\cup x_J\in E_1^{p'(I\sqcup J),|I\sqcup J|-p'(I\sqcup J)}$ is nonzero, represents $u_{I\sqcup J}=\pm u_I\cup u_J$, and is computed as follows. Write $x_I=\sum_{|K|=p'(I)+1}x_{I,K}$ where $x_{I,K}\in H^{|I|-p'(I)}(X_K)$, and similarly $x_J=\sum x_{J,L}$. If $K$ consists of indices $k_0<\cdots<k_{p'(I)}$ and $L$ consists of indices $\ell_0<\cdots<\ell_{p'(J)}$, then up to sign, each $x_{I,K}\cup x_{J,L}$ is the cup product of the restrictions of $x_{I,K}$ and $x_{J,L}$ to $X_{K\cup L}$ if $k_{p'(I)}=\ell_0$, and zero otherwise. The cup product $x_I\cup x_J$ is then computed by bilinearity. Since it is nonzero, some $x_{I,K}\cup x_{J,L}$ with $|K\cap L|=1$ is nonzero. In particular, the restriction map $H^{|I|-p'(I)}(X_K)\to H^{|I|-p'(I)}(X_{K\cup L})$ is nonzero. Since this map factors through $H^{|I|-p'(I)}(X_M)$ for all $K\subset M\subset K\cup L$, we see that any such group is nonzero, showing that
\begin{enumerate}[(i')]
\setcounter{enumi}{2}
\item If $p'(I)+p'(J)=p'(I\sqcup J)$, then $a_{r+1}\ge|I|-p'(I)$ for all $p'(I)\le r\le p'(I)+p'(J)$.
\end{enumerate}
It is easy to check that if $p'\colon 2^{\{1,\cdots,n\}}\to\{0,\cdots,k-1\}$ satisfies (i')(ii')(iii'), then $p\colon\{0,\cdots,n\}\to\{0,\cdots,k-1\}$ defined by $p(i)=\min\{p'(I)\colon|I|=i\}$ satisfies (i)(ii)(iii). 
(Conversely, if $p$ satisfies (i)(ii)(iii), then $p'(I)=p(|I|)$ satisfies (i')(ii')(iii').)
\end{proof}

\begin{Prop}\label{prop:minimal}
The admissible tuples in Examples~\ref{ex:n} through ~\ref{ex:lax} (excluding the general family $\vec a^{(\ell)}$ in Example~\ref{ex:lax}) are all minimal, with the possible exception of Example~\ref{ex:rubinstein_tillmann} for even $n\ge8$.
\end{Prop}
\begin{proof}
Examples~\ref{ex:n} and~\ref{ex:bisect} are easy. For the remaining examples, we first observe that if $p\colon\{0,\cdots,n\}\to\{0,\cdots,k-1\}$ is a support function for a $k$-tuple $\vec a=\vec a^{\ge0}$, the conditions of Theorem~\ref{thm:obstruction} imply
\begin{equation}\label{eq:a_p_bound}
i-a_{p(i)+1}\le p(i)\le\min(i,k-1),\ \forall i.
\end{equation}

Example~\ref{ex:0000}: If $\vec a=\vec a^{\ge0}=(a_1,\cdots,a_k)$ is an $n$-admissible tuple not larger than the $(n+1)$-tuple $(0,\cdots,0)$, taking $i=n$ in \eqref{eq:a_p_bound}, we get $n=n-a_{p(n)+1}\le k-1$. This implies $k\ge n+1$, and hence $\vec a$ is equal to the $(n+1)$-tuple $(0,\cdots,0)$.

Example~\ref{ex:multisection}: If $p$ is a support function for the $(n-1)$-tuple $(1,\cdots,1,2)$, then \eqref{eq:a_p_bound} implies $p(n)=n-2$ and $p(i)=i-1$ or $i$ for all $i<n$. By $p(i)+p(n-i)\le p(n)$, we actually have $p(i)=i-1$ for all $0<i<n$. This completely determines the function $p$, and the minimal tuple supported by $p$ is $(1,\cdots,1,2)$, so indeed any tuple strictly smaller than it is not supported by $p$ and hence not by any function.

Example~\ref{ex:111}: We show that any $(a_1,\cdots,a_k)$ strictly smaller than the $n$-tuple $(1,\cdots,1)$ is not $n$-admissible. By \eqref{eq:a_p_bound}, any support function $p$ of $\vec a$ has $p(i)=i$ or $i-1$ for all $i$, and $p(n)=n-1$. Thus, $k=n$, $a_n=1$. By the superadditivity of $p$, $p(1)=0$, and thus $a_1=1$. By assumption, there exists an index $1<i<n$ with $a_i\le0$. Then by \eqref{eq:a_p_bound}, we must have $p(i)=i$. Let $j$ be the maximal index with $p(j)=j$; thus $i\le j<n$. Since we have $p(1)+p(j)=j=p(j+1)$, by condition (iii) we obtain $a_{r+1}\ge 1-p(1)=1$ for $0\le r\le j$. In particular, $0\ge a_i\ge1$, a contradiction.

Example~\ref{ex:rubinstein_tillmann}(1): If $p$ is a support function for the PL admissible $k$-tuple $\vec a=(1,2,\cdots,k)$, we have $p(i)\ge i-a_{p(i)+1}=i-p(i)-1$, so $p(i)\ge\lfloor i/2\rfloor$. This gives $p(n)=k-1=\lfloor n/2\rfloor$, and by superadditivity of $p$ we obtain $p(i)=\lfloor i/2\rfloor$ for all $i$. One can check that the minimal tuple supported by $p$ is indeed $\vec a=(1,\cdots,k)$, showing the minimality of $\vec a$.

Example~\ref{ex:rubinstein_tillmann}(2) for $n\le6$: For $n=4$, the minimality of $(1,1,2)$ is a special case of Example~\ref{ex:multisection}. For $n=6$, one can check that the unique support function for the $6$-admissible tuple $(1,2,2,3)$ is given by $p(0)=p(1)=0$, $p(2)=p(3)=1$, $p(4)=2$, $p(5)=p(6)=3$, and the minimal tuple supported by the function $p$ is indeed $(1,2,2,3)$, proving its minimality.

Example~\ref{ex:lax}: We show that any $(a_1,\cdots,a_\ell)$ strictly smaller than the $k$-tuple $(n-1,0,\cdots,0,n-k+1)$ is not $n$-admissible. Any support function $p$ for $\vec a$ has $p(n)=k-1$, hence $\ell=k$ and $a_\ell=n-k+1=a_k$. By Observation~\ref{obs:adm_tuples}, we must have $a_i=0$ for $1<i<k$, so $a_1<n-1$ by assumption. This implies by \eqref{eq:a_p_bound} that $p(n-1)=k-1$. Consequently, $p(1)+p(n-1)=k-1=p(n)$, so condition (iii) implies that $a_2\ge1-p(1)=1$, a contradiction.
\end{proof}

\begin{Ex}
For Rubinstein--Tillmann's examples (Example~\ref{ex:rubinstein_tillmann}) with even $n=2k-2\ge8$, however, Theorem~\ref{thm:obstruction} does not imply that the PL $n$-admissible $k$-tuple $(1,2,\cdots,k-2,k-2,k-1)$ is minimal. In fact, writing $m=\lfloor k/2\rfloor$, the strictly smaller $k$-tuple $(1,2,\cdots,m,m,m+1,\cdots,k-1)$ is minimally supported by the function $p$ defined by $$p(i)=\begin{cases}\lfloor i/2\rfloor,&0\le i\le n/2\\\lceil i/2\rceil,&n/2<i\le n.\end{cases}$$ See Figure~\ref{fig:support_function} for a visualization for the case $n=8$.
\end{Ex}

\begin{figure}[ht]
\centering
\begin{tikzpicture}[scale=0.8]
% Grid: 6 by 6
\draw[step=1, gray!40, very thin] (0,0) grid (6,6);
% Axes
\draw[->, thick] (0,0) -- (6.4,0) node[right] {$p$};
\draw[->, thick] (0,0) -- (0,6.4) node[above] {$q$};
% E_1 label
\node[above right] at (6,6) {$E_1$};
% Labels along axes, placed between vertices
\foreach \p in {0,...,5} {
  \node[below] at ({\p+0.5},0) {$\p$};
}
\foreach \q in {0,...,5} {
  \node[left] at (0,{\q+0.5}) {$\q$};
}
% Stars not circled
\foreach \p/\q in {
  1/0,
  2/0,2/1,
  3/0,3/1,
  4/0,4/1,4/2
} {
  \node at ({\p+0.5},{\q+0.5}) {$*$};
}
% Circled stars: two highest in each column
\foreach \p/\q in {
  0/0,0/1,
  1/1,1/2,
  2/2,
  3/2,3/3,
  4/3,4/4
} {
  \node[draw,circle,inner sep=1.5pt] at ({\p+0.5},{\q+0.5}) {$*$};
}
\end{tikzpicture}
\caption{The $5$-tuple $(1,2,2,3,4)$ is minimally supported by the function $\{0,\cdots,8\}\to\{0,\cdots,4\}$ determined by the rule $p+q\mapsto p$ applied to the circled asterisks.}
\label{fig:support_function}
\end{figure}
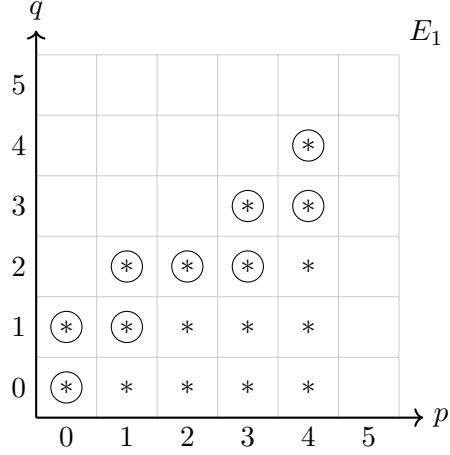

\begin{Que}
For $n=2k-2\ge8$, $m=\lfloor k/2\rfloor$, is the $k$-tuple $(1,2,\cdots,m,m,m+1,\cdots,k-1)$ (PL or smoothly) $n$-admissible?
\end{Que}

\subsubsection{Examples in low dimensions}\label{sbsbsec:ex_lowdim}
We state the consequences of Sections~\ref{sbsbsec:adm_tuples_construction} and~\ref{sbsbsec:adm_tuples_obstruction} in dimensions $n\le6$ toward answering Question~\ref{que:adm_tuples} and Question~\ref{que:tight_adm_tuples} in the smooth (equivalently, in these dimensions, PL) category.

We say a tuple $\vec a$ is \textit{algebraically $n$-allowable} if it admits a support function $p$ as in Theorem~\ref{thm:obstruction}. Thus, (smooth, PL, (co)homological) $n$-admissibility implies algebraic $n$-allowability. 

For each $n$, it is easy to determine the set of minimal algebraically $n$-allowable tuples. For instance, one can enumerate all maps $p\colon\{0,\cdots,n\}\to\Z_{\ge0}$ satisfying (i) and (ii) in Theorem~\ref{thm:obstruction}. For each such $p$, we use (iii) to determine uniquely a tuple $\vec a_p$ minimally supported by $p$. Then, the set of minimal algebraically $n$-allowable tuples is the set of minimal elements in $$\{\vec a_p\colon p\colon\{0,\cdots,n\}\to\Z_{\ge0}\text{ satisfying (i)(ii)}\}.$$ We state below the computer-generated results for $n\le6$. Let $\mathcal A_{n,alg}^{min}$, $\mathcal A_{n,sm}^{min}$ denote the set of minimal algebraically $n$-allowable tuples and (smoothly) $n$-admissible tuples, respectively.

\begin{Ex}[$n\le3$]
$$\mathcal A_{0,alg}^{min}=\{(0)\},$$ $$\mathcal A_{1,alg}^{min}=\{(1),(0,0)\},$$ $$\mathcal A_{2,alg}^{min}=\{(2),(1,1),(0,0,0)\},$$ $$\mathcal A_{3,alg}^{min}=\{(3),(1,2),(1,1,1),(2,0,1),(0,0,0,0)\}.$$
By Examples~\ref{ex:n},~\ref{ex:bisect},~\ref{ex:0000},~\ref{ex:111},~\ref{ex:lax}, we see that every element in $\mathcal A_{n,alg}^{min}$, $n=0,1,2,3$, is indeed an $n$-admissible tuple. Therefore, $\mathcal A_{n,sm}^{min}=\mathcal A_{n,alg}^{min}$ for $n\le3$, answering Question~\ref{que:adm_tuples} in these dimensions.
\end{Ex}

\begin{Ex}[$n=4$]
$$\mathcal A_{4,alg}^{min}=\{(4),(2,3),(1,1,2),(3,0,2),(1,1,1,1),(1,2,0,1),(2,0,1,1),(3,0,0,1),(0,0,0,0,0)\}.$$ By Examples~\ref{ex:n},~\ref{ex:bisect},~\ref{ex:0000},~\ref{ex:multisection},~\ref{ex:111},~\ref{ex:lax}, we see that every element in $\mathcal A_{4,alg}^{min}$ except possibly $(1,2,0,1)$ and $(2,0,1,1)$ is $4$-admissible.

We could not decide if these two extra tuples are $4$-admissible, although we observe that $(2,2,0,1)$ is $4$-admissible by applying Example~\ref{ex:lax} to the $4$-admissible tuple $(2,3)$. Consequently, in view of Observation~\ref{obs:adm_tuples}, to determine the set of $4$-admissible tuples, it is sufficient to determine the $4$-admissibilities of $(1,2,0,1)$ and $(2,0,1,1)$.

In particular, we may conclude that all tight $4$-admissible tuples are $$\mathcal A_{4,sm}^{tight}=\{(4),(2,3),(1,1,2),(1,1,1,1),(0,0,0,0,0)\}.$$
\end{Ex}

\begin{Que}
Are $(1,2,0,1)$ and $(2,0,1,1)$ $4$-admissible?
\end{Que}

\begin{Ex}[$n=5$]
\begin{align*}
\mathcal A_{5,alg}^{min}=&\,\{(5),(2,4),(1,2,3),(3,1,3),(4,0,3),(1,1,1,2),(2,3,0,2),(3,0,2,2),(4,0,0,2),\\&\,(1,1,1,1,1),(1,1,2,0,1),(1,2,0,1,1),(2,0,1,1,1),(2,3,0,0,1),(3,0,0,1,1),\\&\,(3,0,2,0,1),(4,0,0,0,1),(0,0,0,0,0,0)
\}.
\end{align*}
We conclude that the tight $5$-admissible tuples are $$\mathcal A_{5,sm}^{tight}=\{(5),(2,4),(1,2,3),(1,1,1,2),(1,1,1,1,1),(0,0,0,0,0,0)\},$$ realized by Examples~\ref{ex:n} through ~\ref{ex:rubinstein_tillmann} each used exactly once.
\end{Ex}
\begin{Ex}[$n=6$]
\begin{align*}
\mathcal A_{6,alg}^{min}=&\,\{(6),(3,5),(2,2,4),(4,1,4),(5,0,4),(1,2,2,3),(2,3,1,3),(2,4,0,3),(3,1,2,3),\\&\,(4,0,3,3),(4,1,1,3),(5,0,0,3),(1,1,1,1,2),(1,2,3,0,2),(2,3,0,2,2),(2,4,0,0,2),\\&\,(3,0,2,2,2),(3,1,3,0,2),(4,0,0,2,2),(4,0,3,0,2),(5,0,0,0,2),(1,1,1,1,1,1),\\&\,(1,1,1,2,0,1),(1,1,2,0,1,1),(1,2,0,1,1,1),(1,2,3,0,0,1),(2,0,1,1,1,1),\\&\,(2,3,0,0,1,1),(2,3,0,2,0,1),(2,4,0,0,0,1),(3,0,0,1,1,1),(3,0,2,0,1,1),\\&\,(3,0,2,2,0,1),(3,1,3,0,0,1),(4,0,0,0,1,1),(4,0,0,2,0,1),(4,0,3,0,0,1),\\&\,(5,0,0,0,0,1),(0,0,0,0,0,0,0)\}.
\end{align*}
From Examples~\ref{ex:n} through ~\ref{ex:rubinstein_tillmann}, we conclude that the tight $6$-admissible tuples include $(6),(3,5),\allowbreak(1,2,2,3),(1,1,1,1,2),(1,1,1,1,1,1),(0,0,0,0,0,0,0)$, and possibly some additional $3$-tuples.
\end{Ex}
\begin{Que}
Is $(2,2,4)$ $6$-admissible?
\end{Que}
From the examples above, it is tempting to believe that there exists a unique tight $n$-admissible $k$-tuple for each $1\le k\le n+1$. We remark, however, that this pattern may break down. For example, $(1,2,3,4)$ and $(2,2,2,4)$ are both minimal algebraically $7$-allowable $4$-tuples.

\printbibliography

\end{document}